\documentclass[12pt]{amsart}

\usepackage{amssymb}
\usepackage{graphicx}
\usepackage{psfrag}
\usepackage{stmaryrd}
\usepackage{color}

\newcommand\excise[1]{}

\numberwithin{equation}{section}

\newtheorem{thm}{Theorem}[section]
\newtheorem{lemma}[thm]{Lemma}
\newtheorem{prop}[thm]{Proposition}
\newtheorem{cor}[thm]{Corollary}
\newtheorem{prob}[thm]{Open Problem}
\newtheorem{conj}[thm]{Conjecture}

\theoremstyle{definition}
\newtheorem{defn}[thm]{Definition}
\newtheorem{remark}[thm]{Remark}

\newtheorem{example}[thm]{Example}
\newtheorem{exercise}[thm]{Exercise}
\newtheorem{questions}[thm]{Questions}
\newtheorem{answers}[thm]{Answers}

\newcommand\CC{\mathbb{C}}
\newcommand\FF{\mathbb{F}}
\newcommand\NN{\mathbb{N}}
\newcommand\QQ{\mathbb{Q}}
\newcommand\RR{\mathbb{R}}
\newcommand\ZZ{\mathbb{Z}}
\newcommand\bb{\mathbf{b}}
\newcommand\ee{\mathbf{e}}
\newcommand\kk{\Bbbk}
\newcommand\mm{\mathfrak{m}}
\newcommand\uu{\mathbf{u}}
\newcommand\vv{\mathbf{v}}
\newcommand\ww{\mathbf{w}}
\newcommand\pp{\mathfrak{p}}
\newcommand\qq{\mathfrak{q}}
\newcommand\xx{\mathbf{x}}

\newcommand\oJ{{\hspace{.45ex}\overline{\hspace{-.45ex}J}}}
\newcommand\oQ{\hspace{.15ex}\ol{\hspace{-.15ex}Q\hspace{-.25ex}}\hspace{.25ex}}
\newcommand\oT{\ol T}
\newcommand\del{\partial}
\newcommand\til{\mathord\sim}

\newcommand\ttt{\mathbf{t}}
\newcommand\from{\leftarrow}

\newcommand\into{\hookrightarrow}
\newcommand\onto{\twoheadrightarrow}

\newcommand\otno{\twoheadleftarrow}
\newcommand\ffrom{\longleftarrow}
\newcommand\minus{\smallsetminus}
\newcommand\cupdot{\ensuremath{\mathbin{\mathaccent\cdot\cup}}}
\newcommand\goesto{\rightsquigarrow}

\newcommand\Implies{\Longrightarrow}

\newcommand\nothing{\varnothing}

\newcommand\upsupseteq{\mathrel{\rotatebox[origin=c]{90}{$\supseteq$}}}
\newcommand\upequals{\mathrel{\rotatebox[origin=c]{90}{$=$}}}

\newcommand\<{\langle}
\renewcommand\>{\rangle}
\renewcommand\aa{\mathbf{a}}
\renewcommand\tt{{\mathbf{t}}}
\renewcommand\th{\mathrm{th}}
\renewcommand\phi{\varphi}

\renewcommand\iff{\Leftrightarrow}
\renewcommand\implies{\Rightarrow}

\def\1100{$\begin{array}{@{}c@{}}\\[-3.5ex]\scriptscriptstyle
	1100\\[-1.5ex]\scriptscriptstyle 0111\\[-.5ex]\end{array}$}

\newcommand\ol[1]{{\overline{#1}}{}}

\newcommand\reaction[2]{%
	\mathbin{\begin{array}{@{\,}c@{\,}}
		\\[-3.5ex]\scriptstyle#1
		\\[-1.2ex]\rightleftharpoons
		\\[-1.5ex]\scriptstyle#2
	\end{array}}}

\newcommand\sat{\mathrm{sat}}
\newcommand\qdeg{\mathrm{qdeg}}
\newcommand\tdeg{\mathrm{tdeg}}
\newcommand\toral{\text{\rm toral}}
\newcommand\water{\ensuremath{\text{H}_{\text{2}}\text{O}}}
\newcommand\andean{\text{\rm Andean}}
\newcommand\oxygen{\text{O}_{\text{2}}}
\newcommand\peroxide{\ensuremath{\text{H}_{\text{2}}\text{O}_{\text{2}}}}

\DeclareMathOperator\Hom{Hom}
\DeclareMathOperator\std{std}
\DeclareMathOperator\vol{vol}

\newcommand\nim{\textsc{Nim}}
\newcommand\dawson{\textsc{Dawson's Chess}}
\newcommand\chess{\textsc{Chess}}

\newcommand{\aoverb}[2]{{\genfrac{}{}{0pt}{1}{#1}{#2}}}
\def\twoline#1#2{\aoverb{\scriptstyle {#1}}{\scriptstyle {#2}}}

\def\filleftmap{\mathord\leftarrow \mkern-6mu
	\cleaders\hbox{$\mkern-2mu \mathord- \mkern-2mu$}\hfill
	\mkern-6mu \mathord-}
\def\fillmapsfrom{\mathord\leftarrow \mkern-6mu
	\cleaders\hbox{$\mkern-2mu \mathord- \mkern-2mu$}\hfill
	\mkern-6mu\mathord-\mathord{\raisebox{.5pt}{$\!\shortmid$}}}

\definecolor{darkgreen}{rgb}{0,.5,0}
\definecolor{darkpurple}{rgb}{.6,0,.6}
\definecolor{turquoise}{rgb}{0,.5,.7}

\begin{document}

\mbox{}\vspace{-3ex}
\title
{Theory and applications of lattice point methods for binomial ideals}
\author{Ezra Miller}
\address{Mathematics Department\\Duke University\\Durham, NC 27707}
\email{ezra@math.duke.edu}

\begin{abstract}
This survey of methods surrounding lattice point methods for binomial
ideals begins with a leisurely treatment of the geometric
combinatorics of binomial primary decomposition.  It then proceeds to
three
independent applications whose motivations come from outside of
commutative algebra: hypergeometric systems, combinatorial game
theory, and chemical dynamics.  The exposition is aimed at students
and researchers in algebra; it includes many examples, open problems,
and elementary introductions to the motivations and background from
outside of~algebra.\vspace{-1ex}
\end{abstract}


\date{8 September 2010}

\maketitle

\tableofcontents

\vspace{-7ex}
\section*{Introduction}\label{s:intro}

Binomial ideals in polynomial rings over algebraically closed fields
admit \emph{binomial primary decompositions}: expressions as
intersections of primary binomial ideals.  The algebra of these
decompositions is governed by the geometry of lattice points in
polyhedra and related lattice-point combinatorics arising from
congruences on commutative monoids.  The treatment of this geometric
combinatorics is terse at the source \cite{primDecomp}.  Therefore, a
primary goal of this exposition is to provide a more leisurely tour
through the relevant phenomena; this is the concern of
Sections~\ref{s:affine}, \ref{s:primary}, and~\ref{s:decomp}.

That the geometry of congruences should govern binomial primary
decomposition was a realization made in the context of classical
multivariate hypergeometric series, going back to Horn, treated in
Section~\ref{s:horn}.  Lattice-point combinatorics related to monoids
and congruences has recently been shown relevant to the theory of
combinatorial games, and is sure to play a key role in algorithms for
computing rational strategies and mis\`ere quotients, as discussed in
Section~\ref{s:games}.  Finally, binomial commutative algebra is
central to a long-standing conjecture on the dynamics of chemical
reactions under mass-action kinetics.  The specifics of this
connection are briefly outlined in Section~\ref{s:chem}, along with
the potential relevance of combinatorial methods for binomial primary
decomposition.  Limitations of time and space prevented the inclusion
of algebraic statistics in this survey; for an exposition of binomial
aspects of Markov bases and conditional independence models, such as
graphical models, as well as applications to phylogenetics, see
\cite{algStat09} and the references therein.

Sections~\ref{s:affine}, \ref{s:primary}, and~\ref{s:decomp} are
complete in the sense that statements are made in full generality, and
precise references are provided for the details of any argument that
is only sketched.  In contrast, Sections~\ref{s:horn}, \ref{s:games},
and~\ref{s:chem} are more expository.  The results there are sometimes
stated in less than full generality---but still mathematically
precisely---to ease the~exposition.  In addition,
Sections~\ref{s:horn}, \ref{s:games}, and~\ref{s:chem} are independent
of one another, and to a large extent independent of
Sections~\ref{s:affine}, \ref{s:primary}, and~\ref{s:decomp}, as well;
readers interested in the applications should proceed to the relevant
sections and refer back~as~necessary.

\medskip
\noindent
\textbf{Acknowledgements.}  I am profoundly grateful to the organizers
and participants of the International School on Combinatorics at
Sevilla, Spain in January 2010, where these notes were presented as
five lectures.  That course was based on my Abel Symposium talk at
Voss, Norway in June 2009; I am similarly indebted to that meeting's
organizers.  Thanks also go to my coauthors, from whom I learned so
much while working on various projects mentioned in this survey.
Funding was provided by NSF CAREER grant DMS-0449102 = DMS-1014112 and
NSF grant DMS-1001437.

\part{Theory}

\section{Affine semigroups and prime binomial ideals}\label{s:affine}%

\subsection{Affine semigroups}

Let $\ZZ^d \subset \RR^d$ denote the integer points in a real vector
space of dimension~$d$.  Any integer point configuration
$$%
  A = \{\aa_1,\ldots,\aa_n\} \subset \ZZ^d
  \quad\longleftrightarrow\quad
  A = \left[
      \begin{array}{ccc}
	  |   &        &   |   \\
	\aa_1 & \cdots & \aa_n \\
	  |   &        &   |   
      \end{array}
      \right]
  \in \ZZ^{d \times n}
$$
can be identified with a $d \times n$ integer matrix.

\begin{defn}
A \emph{monoid} is a set with an associative binary operation and an
identity element.  An \emph{affine semigroup} is a monoid that is
isomorphic to
$$%
  \NN A = \NN\{\aa_1,\ldots,\aa_n\}
        = \{c_1\aa_1 + \cdots + c_n\aa_n \mid c_1,\ldots,c_n \in \NN\}
$$
for some lattice point configuration $A \subset \ZZ^d$.
\end{defn}

Thus a monoid is a group without inverses.  Although a
\emph{semigroup} is generally not required to have an identity
element, standard terminology from the literature dictates that an
affine semigroup is a monoid, and in particular (isomorphic to) a
finitely generated submonoid of an integer lattice~$\ZZ^d$ for
some~$d$.

\begin{example}\label{e:0123}
The configuration
$$%
\begin{array}{cc}
\\[-6ex]
  \begin{array}{c}
  \\\\\\\\
  A =
  \left[
  \begin{array}{cccc}
	 1 & 1 & 1 & 1 \\
	 0 & 1 & 2 & 3
       \end{array}
  \right]
  \qquad\longleftrightarrow\ \
  \end{array}
&
  \begin{array}{c}\includegraphics{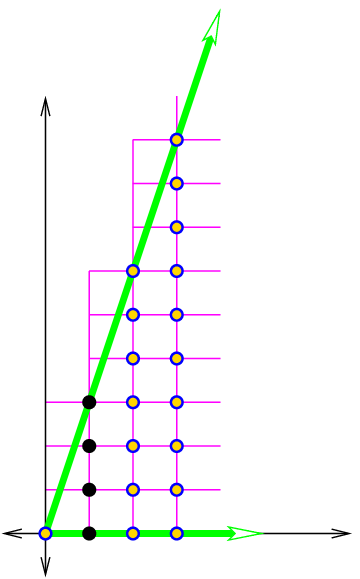}\end{array}
\\[-1ex]
\end{array}
$$
in~$\ZZ^2$, drawn as solid dots in the plane, generates the affine
semigroup comprising all lattice points in the real cone bounded by
the thick horizontal ray and the diagonal ray.  This example will
henceforth be referred to as ``\,0123\,''.
\end{example}

\begin{example}\label{e:1100}
A point configuration is allowed to have repeated elements, such as
$$%
\begin{array}{cc}
  \begin{array}{c}
  \\
  A =
  \left[
  \begin{array}{cccc}
	 1 & 1 & 0 & 0 \\
	 0 & 1 & 1 & 1
       \end{array}
  \right]
  \qquad\longleftrightarrow\ \
  \end{array}
&
  \begin{array}{c}\includegraphics{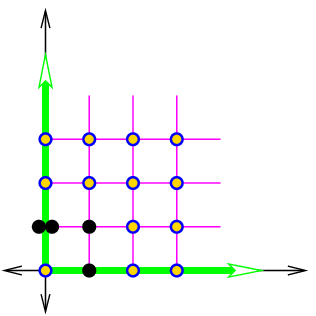}\end{array}
\end{array}
$$
in~$\ZZ^2$, which generates the affine semigroup $\NN A = \NN^2$ of
all lattice points in the nonnegative quadrant.  This example will
henceforth be referred to as ``\,\1100\,''.
\end{example}

\begin{example}\label{e:hartcone}
It will be helpful, later on, to have a three-dimensional example
ready.  Consider a square at height~$1$ parallel to the horizontal
plane.  It can be represented as a matrix and point configuration
in~$\ZZ^3$ as follows:
$$%
\psfrag{a}{$\scriptstyle a$}
\psfrag{b}{$\scriptstyle b$}
\psfrag{c}{$\scriptstyle c$}
\psfrag{d}{$\scriptstyle d$}
\begin{array}{cc}
  \begin{array}{c}
  \\
  A =
  \begin{array}{@{}c@{}}
  \begin{array}{cccc}
	 a & b & c & d
  \end{array}\\[.75ex]
  \left[
  \begin{array}{cccc}
	 0 & 1 & 0 & 1 \\
	 0 & 0 & 1 & 1 \\
	 1 & 1 & 1 & 1
  \end{array}
  \right]\\
  \begin{array}{cccc}
	   &   &   &  
  \end{array}
  \end{array}
  \qquad\longleftrightarrow\ \
  \end{array}
&
  \begin{array}{c}\includegraphics{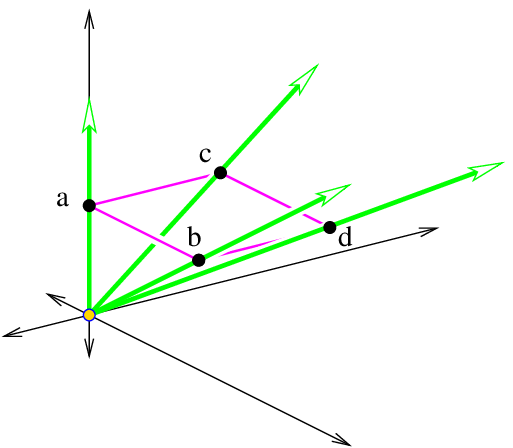}\end{array}
\end{array}
$$
The affine semigroup $\NN A$ comprises all of the lattice points in
the real cone generated by the vertices of the square.
\end{example}

In all of these examples, the affine semigroups are \emph{normal}:
each one equals the set of all lattice points from a rational
polyhedral cone.  General affine semigroups need not be normal, though
they always comprise ``most'' of the lattice points in~a~cone.

\begin{example}
In Example~\ref{e:0123}, the outer columns of the matrix, namely
$\left[\twoline 10\right]$ and $\left[\twoline 13\right]$, correspond
to the extremal rays; they therefore generate the rational polyhedral
cone whose lattice points constitute the 0123 affine semigroup.
Consequently, the configuration $\left[\twoline 10\twoline 12\twoline
13\right]$ generates the same rational polyhedral cone, but the affine
semigroup it generates is different---and not normal---because the
point $\left[\twoline 11\right]$ does not lie in it, even though the
configuration still generates~$\ZZ^2$.
\end{example}

The geometry of binomial primary decomposition is based on the sort of
geometry that arises from the projection determined by~$A$.  To be
more precise, $A$ determines a monoid morphism $\ZZ^d \from \NN^n$.
This morphism can be expressed as the restriction of the group
homomorphism (the linear map) $\ZZ^d \from \ZZ^n$ induced by~$A$.  The
diagram is as follows:
$$%
\begin{array}{c@{\ }c@{\ }c@{\ }c@{\ }c@{\ }c@{\ }c}
        \aa_j & \fillmapsfrom & \ee_j
\\            &\ \ \begin{array}{|c|}
               \hline \scriptstyle n\\\scriptstyle d\hfill\\[-2ex]
               \qquad A\qquad\\[2ex]\hline
               \end{array}\ \ &
\\
       \ZZ^d    & \filleftmap &    \NN^n
\\
     \upequals  &             & \upsupseteq &
\\[.5ex]
       \ZZ^d    & \filleftmap &    \ZZ^n    & \ffrom &      L      & \from & 0
\\
    \upsupseteq &             & \upsupseteq &        & \upsupseteq &       & 
\\[.5ex]
       \RR^d    & \filleftmap &    \RR^n    & \ffrom &   \RR L     & \from & 0
\end{array}
$$
The kernel~$L$ of the homomorphism $\ZZ^d \from \ZZ^n$ induced by~$A$
is a \emph{saturated} lattice in~$\ZZ^n$, meaning that $\ZZ^n/L$ is
torsion-free, or equivalently that $L = \RR L \cap \ZZ^n$, where $\RR
L = \RR \otimes_\ZZ L$ is the real subspace of~$\RR^n$ generated
by~$L$.

It makes little sense to say that the monoid morphism $\ZZ^d \from
\NN^n$ has a kernel: it is often the case that $L \cap \NN^n = \{0\}$,
even when the monoid morphism is far from injective.  However, when
$\ZZ^d \from \NN^n$ fails to be injective, the fibers admit clean
geometric descriptions, inherited from the fact that the fibers of the
vector space map $\RR^d \from \RR^n$ are the cosets of~$\RR L$
in~$\RR^n$.

\begin{defn}\label{d:polyhedron}
A \emph{polyhedron} in a real vector space is an intersection of
finitely many closed real half-spaces.  
\end{defn}

This survey assumes basic knowledge of polyhedra.  Readers for whom
Definition~\ref{d:polyhedron} is not familiar are urged to consult
\cite[Chapters~0, 1, and~2]{ziegler}.

\begin{lemma}\label{l:fiber}
The fiber of the monoid morphism $\ZZ^d \stackrel A\from \NN^n$ over a
given lattice point $\alpha \in \ZZ^d$ is the set
$$%
  F_\alpha = \NN^n \cap (\uu + L) = \NN^n \cap P_\alpha
$$%
of lattice points in the polyhedron
$$%
  P_\alpha = (\uu + \RR L) \cap \RR^n_{\geq 0}
$$
for any vector $\uu \in \ZZ^d$ satisfying $A\uu = \alpha$, where $L =
\ker A$.
\end{lemma}

This description is made particularly satisfying by the fact that the
polyhedra for various $\alpha \in \ZZ^d$ are all related to one
another.

\begin{example}\label{e:triangles}
When $A = [1\ \,1\ \,1]$ is the ``coordinate-sum'' map $\NN \from
\NN^3$, the polyhedra $P_\alpha$ for $\alpha \in \NN$ are equilateral
triangles, the lattice points in them corresponding to the monomials
of total degree~$\alpha$ in three variables:
$$%
\begin{array}{cc}
  \begin{array}{c}
    \\
    \text{fibers of } A = [1\ \,1\ \,1]
    \qquad\longleftrightarrow\ \
  \end{array}
&
  \begin{array}{c}\includegraphics{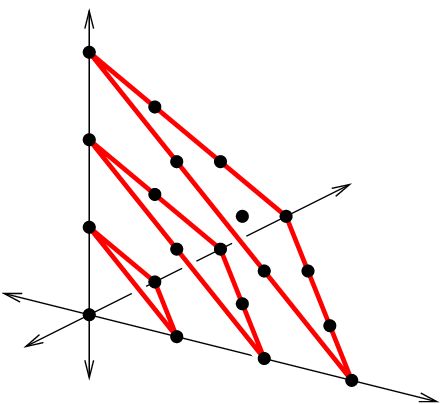}\end{array}
\end{array}
$$
\end{example}

The polyhedra in Example~\ref{e:triangles} are all scalar multiples of
one another, but this phenomenon is special to codimension~$1$.  In
general, when \mbox{$n - d > 1$}, the polyhedra $P_\alpha$ in the
family indexed by $\alpha \in \NN A$ have facet normals chosen from
the same fixed set of possiblities---namely, the images in~$L^*$ of
the dual basis vectors of~$(\ZZ^n)^*$---so their shapes feel roughly
similar, but faces can shrink or~disappear.

\begin{example}
In the case of 0123, a basis for the kernel $\ker A$ can be chosen so
that the inclusion $\ZZ^n \from \ker A$ is given by the matrix
$$%
\begin{array}{cc}
  \begin{array}{c}
  \\[-4ex]
  B =
  \left[
  \begin{array}{@{\,}rr}
	 1 & 0 \\
	-2 & 1 \\
	 1 &-2 \\
	 0 & 1 
       \end{array}
  \right]
  \qquad\goesto\ \
  \end{array}
&
  \psfrag{,}{$,$}
  \psfrag{u1}{$P_\alpha$}
  \psfrag{u2}{$P_\beta$}
  \begin{array}{c}\includegraphics{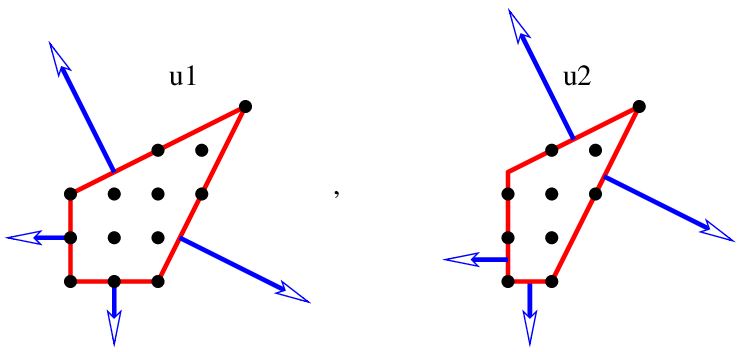}\end{array}
\end{array}
$$
so that, for example, the depicted polytopes $P_\alpha$ and~$P_\beta$
for $\alpha = \left[\twoline{8}{12}\right]$
and $\beta = \left[\twoline{7}{12}\right]$
both have outer normal vectors that are the negatives of the rows
of~$B$.  Moving to $P_\gamma$ for $\gamma =
\left[\twoline{6}{12}\right]$ would shrink the bottom edge entirely.
The corresponding fibers~$F_\alpha$, $F_\beta$, and~$F_\gamma$
comprise the lattice points in these polytopes.
\end{example}

\subsection{Affine semigroup rings}

\begin{defn}
The \emph{affine semigroup ring} of~$\NN A$ over a field~$\kk$ is
$$%
  \kk[\NN A] = \bigoplus_{\alpha\in\NN A} \kk \cdot \tt^\alpha,
$$
a subring of the Laurent polynomial ring
$$%
  \kk[\ZZ^d] = \kk[t_1^{\pm1},\ldots,t_d^{\pm1}],
$$
in which
$$%
  \tt^\alpha = t_1^{\alpha_1}\cdots t_d^{\alpha_d}
  \quad\text{and}\quad \tt^{\alpha+\beta} = \tt^\alpha\tt^\beta.
$$
\end{defn}
The definition could be made with $\kk$ an arbitrary commutative ring,
but in fact the case we care about most is $\kk = \CC$, the field of
complex numbers.  The reason is that the characteristic zero and
algebraically closed hypotheses enter at key points; these notes
intend to be precise about which hypotheses are needed where.

\begin{defn}\label{d:surj}
Denote by $\pi_A$ the surjection
\begin{align*}
  \kk[\NN A] &\stackrel{\pi_A}\otno \kk[\xx]
\\ t^{\aa_i} &\mapsfrom x_i
\end{align*}
onto the affine semigroup ring~$\kk[\NN A]$ from the polynomial
ring~$\kk[\xx] := \kk[x_1\ldots,x_n]$.
\end{defn}

The next goal is to calculate the kernel $I_A = \ker(\pi_A)$.  To do
this it helps to note that both $\kk[\NN A]$ and the polynomial ring
are graded, in the appropriate sense.

\begin{defn}\label{d:A-graded}
Let $A \in \ZZ^{d \times n}$, so $\ZZ A \subseteq \ZZ^d$ is a
subgroup.  A ring $R$ is \emph{$A$-graded} if $R$ is a direct sum of
\emph{homogeneous components}
$$%
  R = \bigoplus_{\alpha\in\ZZ A} R_\alpha, \quad\text{such that}\quad
  R_\alpha R_\beta \subseteq R_{\alpha+\beta}.
$$
An ideal in an $A$-graded ring is \emph{$A$-graded} if it is generated
by homogeneous elements.
\end{defn}

\begin{example}\label{e:A-graded}
The affine semigroup ring $\kk[\NN A]$ is $A$-graded, with $\kk[\NN
A]_\alpha = \kk\cdot\tt^\alpha$ if $\alpha \in \NN A$, and $\kk[\NN
A]_\alpha = 0$ otherwise.  The polynomial ring $\kk[\xx]$ is also
$A$-graded, with
$$%
  \kk[\xx]_\alpha = \kk\cdot\{F_\alpha\},
$$
the vector space spanned by the fiber $F_\alpha$ from
Lemma~\ref{l:fiber}.
\end{example}

\begin{prop}
The kernel of the surjection~$\pi_A$ from Definition~\ref{d:surj} is
\begin{align*}
  I_A &= \<\xx^\uu-\xx^\vv \mid \uu,\vv\in\NN^n\text{ and }A\uu = A\vv\>
\\    &= \<\xx^\uu-\xx^\vv \mid \uu,\vv\in\NN^n\text{ and }\uu-\vv\in \ker A\>.
\end{align*}
\end{prop}
\begin{proof}
The ``$\supseteq$'' containment follows simply because $\pi_A(\xx^\uu)
= \xx^{A\uu}$.  The reverse containment uses the $A$-grading: in the
ring $R = \kk[\xx]/\<\xx^\uu-\xx^\vv\mid A\uu=A\vv\>$, the dimension
of the image of $\kk[\xx]_\alpha$ as a vector space over~$\kk$ is
either $0$ or~$1$ since $A\uu = A\vv = \alpha$ if $\uu$ and~$\vv$ lie
in the same fiber~$F_\alpha$.  On the other hand, $R \onto \kk[\NN
A]_\alpha$ maps surjectively onto the affine semigroup ring by the
``$\supseteq$'' containment already proved.  The surjection must be an
isomorphism because, as we noted in Example~\ref{e:A-graded}, $\kk[\NN
A]_\alpha$ has dimension~$1$ whenever $F_\alpha$ is nonempty.
\end{proof}

\begin{cor}
The \emph{toric ideal} $I_A$ is prime.
\end{cor}
\begin{proof}
$\kk[\NN A]$ is an integral domain, being contained in $\kk[\ZZ^d]$.
\end{proof}

\begin{example}\label{e:0123'}
In the 0123 case, using variables $\{a,b,c,d\}$ instead of
$\{x_1,x_2,x_3,x_4\}$,
$$%
  A =
  \left[
  \begin{array}{cccc}
	 1 & 1 & 1 & 1 \\
	 0 & 1 & 2 & 3
       \end{array}
  \right]
  \text{ and }\
  B =
  \left[
  \begin{array}{@{\,}rr}
	 1 & 0 \\
	-2 & 1 \\
	 1 &-2 \\
	 0 & 1 
       \end{array}
  \right]
  \ \Implies\
  I_A = \<ac-b^2,bd-c^2,ad-bc\>,
$$
where $\ker A$ is the image of~$B$ in~$\ZZ^4$.  The presence of the
binomials $ac-b^2$ and $bd-c^2$ in~$I_A$ translate the statement ``the
columns of the matrix~$B$ lie in the kernel of~$A$''.  However, note
that $I_A$ is not generated by these two binomials; it is a
complicated problem, in general, to determine a minimal generating set
for~$I_A$.
\end{example}

\begin{example}\label{e:1100'}
In the \1100 case, using variables $\{a,b,c,d\}$ instead of
$\{x_1,x_2,x_3,x_4\}$,
$$%
  A =
  \left[
  \begin{array}{cccc}
	 1 & 1 & 0 & 0 \\
	 0 & 1 & 1 & 1
       \end{array}
  \right]
  \text{ and }\
  B =
  \left[
  \begin{array}{@{\,}rr}
	 1 & 1 \\
	-1 &-1 \\
	 1 & 0 \\
	 0 & 1 
       \end{array}
  \right]
  \ \Implies\
  I_A = \<ac-b,c-d\>,
$$
where again $\ker A$ is the image of~$B$ in~$\ZZ^4$.  In this case,
$I_A$ is indeed generated by two binomials corresponding to a basis
for the kernel of~$A$, but not the given basis appearing as the
columns of~$B$.  In Section~\ref{s:horn}, we shall be interested in
the ideal generated by the two binomials corresponding to the columns
of~$B$.
\end{example}

\begin{example}
In the square cone case, using $\{a,b,c,d\}$ instead of
$\{x_1,x_2,x_3,x_4\}$,
$$%
  A =
  \left[
  \begin{array}{cccc}
	 0 & 1 & 0 & 1 \\
	 0 & 0 & 1 & 1 \\
	 1 & 1 & 1 & 1
       \end{array}
  \right]
  \text{ and }\
  B =
  \left[
  \begin{array}{@{\,}rr}
	 1 \\
	-1 \\
	-1 \\
	 1 
       \end{array}
  \right]
  \ \Implies\
  I_A = \<ad-bc\>,
$$
where $\ker A$ is the image of~$B$.  In the codimension~$1$ case, when
$\ker A$ has rank~$1$, the toric ideal is always principal, just as
any codimension~$1$ prime ideal in~$\kk[\xx]$~is.
\end{example}

\subsection{Prime binomial ideals}

At last it is time to define binomial ideals precisely.

\begin{defn}
An ideal $I \subseteq \kk[\xx]$ is a \emph{binomial ideal} if it is
generated by \emph{binomials}
$$%
  \xx^\uu - \lambda\xx^\vv \quad\text{with}\quad \uu,\vv \in \NN^n
  \text{ and } \lambda \in \kk.
$$
\end{defn}

Note that $\lambda = 0$ is allowed: monomials are viable generators of
binomial ideals.  This may seem counterintuitive, but it is forced by
allowing arbitrary nonzero constants~$\lambda$, and in any case, even
ideals generated by differences of monomials (``pure-difference
binomials'') have associated primes containing monomials.

\begin{example}
The ideal $\<x^3-y^2, x^3-2y^2\> \subseteq \kk[x,y]$ is generated by
binomials but equals the monomial ideal $\<x^3,y^2\>$, no matter the
characteristic of~$\kk$.  Worse, the ideal $I = \<x^2 - xy, xy -
2y^2\> \subseteq \kk[x,y]$ is generated by ``honest'' binomials that
are not linear combinations of monomials in~$I$, and yet $I$ contains
monomials, because $I$ contains both of $x^2y - xy^2$ and $x^2y -
2xy^2$, so the monomials $x^2y$ and~$xy^2$ lie in~$I$.
\end{example}

\begin{example}
The pure-difference binomial ideal $I = \<x^2-xy, xy-y^2\> \subseteq
\kk[x,y]$ has a monomial associated prime ideal $\<x,y\>$, since $I =
\<x-y\> \cap \<x^2,y\>$.
\end{example}

In general, which binomial ideals are prime?  We have seen that toric
ideals~$I_A$ are prime, but for binomial primary decomposition in
general it is important to know all of the other binomial primes, as
well.  The answer was given by Eisenbud and Sturmfels
\cite[Corollary~2.6]{binomialIdeals}.

\begin{thm}\label{t:prime}
When $\kk$ is algebraically closed, a binomial ideal $I \subseteq
\kk[x_1,\ldots,x_m]$ is prime if and only if it is the kernel of a
surjective $A$-graded homomorphism $\pi: \kk[\xx] \onto \kk[\NN A]$ in
which the variables are homogeneous.
\end{thm}

The surjection $\pi$ in Theorem~\ref{t:prime} need not equal~$\pi_A$.
For example, $\pi(x_j) = 0$ is allowed, so $I$ could contain
monomials.  Furthermore, even when $\pi(x_j) \neq 0$, the image
of~$x_j$ could be $\lambda\tt^\alpha$ for some $\lambda \neq 1$.
Thus, if $\xx^\uu \mapsto \lambda\tt^\alpha$ and $\xx^\vv \mapsto
\mu\tt^\alpha$, then $\xx^\uu - \frac\mu\lambda\xx^\vv \in I$: the
binomial generators of~$I$ need not be pure differences.  However,
assuming $I$ is prime, we are not free to assign the coefficients
$\lambda$ and~$\mu$ at will:
\begin{align*}\textstyle
  \xx^\uu - \frac\mu\lambda\xx^\vv \in I\ \Implies&\ \ \textstyle
  \xx^{\uu+\ww}-\frac\mu\lambda\xx^{\vv+\ww}\in I\text{ for }\ww\in\NN^n
\\\text{and}&\ \ \textstyle
  \xx^{r\uu} - \frac{\mu^r}{\lambda^r}\xx^{r\vv} \in I\text{ for }r\in\NN.
\end{align*}
The first line means that the coefficient on $\xx^\vv$ in
$\xx^\uu-\xx^\vv$ depends only on $\uu-\vv$, and the second means
essentially that the assignment $\uu-\vv \mapsto \mu/\lambda$
constitutes a homomorphism $\ker A \to \kk^*$.  The precise statement
requires a definition.

\begin{defn}\label{d:IrhoJ}
A \emph{character} on a sublattice $L \subseteq \ZZ^n$ is a
homomorphism $\rho: L \to \kk^*$.  If $L \subseteq \ZZ^J$ for some
subset $J \subseteq \{1,\ldots,n\}$, then
\begin{align*}
  I_{\rho,J} &= I_\rho + \mm_J,
\\\text{where }\quad\ \,
  I_\rho &= \<\xx^\uu - \rho(\uu-\vv)\xx^\vv \mid \uu-\vv \in L\>
\\\text{and }\quad
  \mm_J &= \<x_i \mid i \not\in J\>.
\end{align*}
\end{defn}

\begin{cor}\label{c:prime}
A binomial ideal $I \subseteq \kk[\xx]$ with $\kk$ algebraically
closed is prime if and only if it is $I_{\rho,J}$ for a character
$\rho: L \to \kk^*$ defined on a saturated sublattice $L \subseteq
\ZZ^J$.
\end{cor}

In other words, every prime binomial ideal in the polynomial ring
$\kk[\xx]$ over an algebraically closed field~$\kk$ is toric after
forgetting some of the variables (those outside of~$J$) and rescaling
the rest (by the character~$\rho$).

\begin{remark}
Given any sublattice $L \subseteq \ZZ^n$, a character $\rho$ is
defined as a homomorphism $L \to \kk^*$.  On the other hand, rescaling
the variables $x_j$ for $j \in J$ amounts to a homomorphism $\ZZ^J \to
\kk^*$.  When $L \subsetneq \ZZ^J$, there is usually no unique way to
extend $\rho$ to a character $\ZZ^J \to \kk^*$ (there can be a unique
way if $\kk$ has positive characteristic).  However, there is always
at least one way when $L$ is saturated---so the inclusion $L \into
\ZZ^n$ is split---because the natural map $\Hom(\ZZ^J,\kk^*) \to
\Hom(L,\kk^*)$ is surjective.
\end{remark}

\begin{example}
Let $\omega = \frac{1+\sqrt{-3}}2 \in \CC$ be a primitive cube root
of~$1$.  If $L \subseteq \ZZ^4 \subseteq \ZZ^5$ is spanned by the
columns of the matrix~$B$, below, and the character $\rho$ takes the
indicated values on these generators of~$L$, then $I_{\rho,J}$ for $J
= \{1,2,3,4\}$ is as indicated.
$$%
\begin{array}{@{}r@{}}
  B =
  \left[
  \begin{array}{@{\,}rrr}
	-1 &-1 & 0 \\
	 1 & 2 &-1 \\
	 1 &-1 & 2 \\
	-1 & 0 &-1 \\
	 0 & 0 & 0
       \end{array}
  \right]
\\
\\[-2ex]
  \rho:
  \begin{array}{@{\,}rrr}
	\omega^2\ & \omega\ \,& \omega\ \
  \end{array}
\end{array}
\ \Implies\
I_{\rho,\{1,2,3,4\}} = \<bc-\omega^2ad,b^2-\omega ac,c^2-\omega bd,e\>.
$$
For instance, when $\uu = (0,1,1,0)$ and $\vv = (1,0,0,1)$, we get
$\rho(\uu-\vv) = \omega^2$.  Compare this example to the 0123 case in
Example~\ref{e:0123'}.
\end{example}

\section{Monomial ideals and primary binomial ideals}\label{s:primary}

The lattice-point geometry of binomial primary decompostion
generalizes the geometry of monomial ideals.  For primary binomial
ideals, the connection is particularly clear.  To highlight it, this
section discusses what it looks like for a primary binomial ideal to
have a monomial associated prime.  The material in this section is
developed in the context of an arbitrary affine semigroup ring,
because that generality will be crucial in the applications to
binomial ideals in polynomial rings~$\kk[\xx]$.  In return for the
generality, there are no restrictive hypotheses on the characteristic
or algebraic closure of~$\kk$ to contend with; except in
Example~\ref{e:rhoJ} and Theorem~\ref{t:algbinom}, $\kk$ can
be~artbitrary.

\begin{defn}
An ideal $I \subseteq \kk[Q]$ in the monoid algebra of an affine
semigroup~$Q$ over an arbitrary field is a \emph{monomial ideal} if it
is generated by \emph{monomials} $\tt^\alpha$, and $I$ is a
\emph{binomial ideal} if it is generated by \emph{binomials}
$\tt^\alpha - \lambda\tt^\beta$ with $\alpha,\beta \in Q$ and $\lambda
\in \kk$.
\end{defn}

\subsection{Monomial primary ideals}

Given a monomial ideal, it is convenient to have terminology and
notation for certain sets of monomials and lattice points.
\begin{defn}
If $I \subseteq \kk[Q]$ is a monomial ideal, then write $\std(I)$ for
the set of exponent vectors on its \emph{standard monomials}, meaning
those outside of~$I$.
\end{defn}


\begin{example}\label{e:monomPrimary}
Here is a monomial ideal in $\kk[Q]$ for $Q = \NN^2$:
$$%
\psfrag{x}{$\scriptstyle x$}
\psfrag{z}{$\scriptstyle z$}
\begin{array}{cc}
  \begin{array}{c}
  I = \<x^6, x^5z, x^2z^3,z^4\> \subseteq \kk[x,z]
  \qquad\longleftrightarrow\ \
  \end{array}
&
  \begin{array}{c}\includegraphics{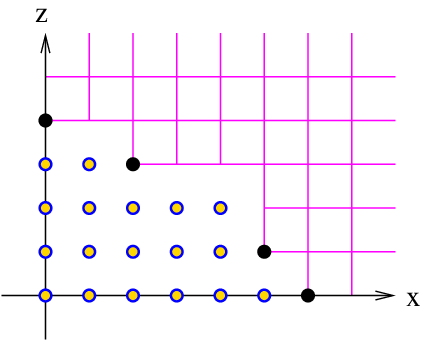}\end{array}
\end{array}
$$
The reason for using $z$ instead of~$y$ will become clear in
Example~\ref{e:monomPrimary'}.  The bottom of the cross-hatched region
is the \emph{staircase} of~$I$; its lower corners are the (lattice
points corresponding to) the generators of~$I$.  The lattice points
below the staircase correspond to the standard monomials of~$I$.
\end{example}

\begin{example}\label{e:monomPrimary'}
The same monomial generators can result in a higher-dimensional
picture if the ambient monoid is different.  Consider the generators
from Example~\ref{e:monomPrimary} but in $\kk[Q]$ for $Q = \NN^3$:
$$%
\psfrag{x}{$\scriptstyle x$}
\psfrag{y}{$\scriptstyle y$}
\psfrag{z}{$\scriptstyle z$}
\begin{array}{cc}
  \begin{array}{c}
  I = \<x^6, x^5z, x^2z^3,z^4\> \subseteq \kk[x,y,z]
  \qquad\longleftrightarrow\ \
  \end{array}
&
  \begin{array}{c}\includegraphics{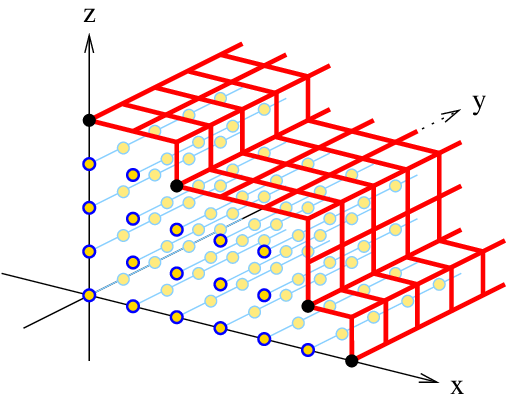}\end{array}
\end{array}
$$
The ``front face'' of the picture---corresponding to the
$xz$-plane---coincides with Example~\ref{e:monomPrimary}.  The surface
cross-hatched in thick lines is the staircase of~$I$; its minimal
elements again correspond to the generators of~$I$, drawn as solid
dots.  Below the staircase sit the lattice points corresponding to the
standard monomials of~$I$.  There are infinitely many standard
monomials, but they occur along finitely many rays parallel to the
$y$-axis, each emanating from a point drawn as a bold hollow dot;
these are the points in~$\std(I)$ in the $xz$-plane.
\end{example}

\begin{example}\label{e:face}
In the square cone case, Example~\ref{e:hartcone}, let $a,b,c,d$ be
the generators of the affine semigroup ring, as indicated in the
figure there.  Thus $\kk[Q] = \kk[a,b,c,d]/\<ad-bc\>$.  The monomial
ideal $\<c,d\> \subseteq \kk[Q]$ is prime; in fact, the composite map
$\kk[a,b] \into \kk[Q] \onto \kk[Q]/\<c,d\>$ is an isomorphism.
\end{example}

The phenomenon in Example~\ref{e:face} is general; the statement
requires a definition.

\begin{defn}
A \emph{face} $F \subseteq Q$ of an affine semigroup $Q \subseteq
\ZZ^d$ is a subset $F = Q \cap H$ obtained by intersecting~$Q$ with a
halfspace $H \subseteq \RR^d$ such that $Q \subseteq H^+$ is contained
in one of the two closed halfspaces $H^+$, $H^-$ defined by~$H$
in~$\RR^d$.
\end{defn}

\begin{lemma}\label{l:face}
A monomial ideal~$I$ in an affine semigroup ring is prime $\iff
\std(I)$~is~a face.  The prime $\pp_F$ for $F \subseteq Q$ induces an
isomorphism $\kk[F] \into \kk[Q] \onto \kk[Q]/\pp_F$.
\end{lemma}

For a proof of the lemma, and lots of additional background on the
connections between faces of cones and the algebra of affine semigroup
rings, see \cite[\S7.2]{cca}.

Before jumping to the question of when an arbitrary binomial ideal is
primary, let us first consider the monomial case.  In a polynomial
ring~$\kk[\xx]$, there is an elementary algebraic description as well
as a satisfying geometric one.  Both will be important in later
sections, but it is the geometric description that generalizes most
easily to arbitrary affine semigroup rings.

\begin{prop}\label{p:monomPrimary}
A monomial ideal $I \subseteq \kk[\xx]$ is primary if and only if
$$%
  I = \<x_{i_1}^{m_1}, \ldots, x_{i_r}^{m_r}, \text{ some other
  monomials in } x_{i_1}, \ldots, x_{i_r}\>.
$$
A monomial ideal $I \subseteq \kk[Q]$ for an arbitrary affine
semigroup $Q \subseteq \ZZ^d$ is $\pp_F$-primary for a face $F
\subseteq Q$ if and only if there are elements
$\alpha_1,\ldots,\alpha_\ell \in Q$ with
$$%
  \std(I) = \bigcup_{k=1}^\ell (\alpha_k + \ZZ F) \cap Q,
$$
where $\ZZ F$ is the subgroup of~$\ZZ^d$ generated by~$F$.
\end{prop}
\begin{proof}
The statement about $\kk[\xx]$ is a standard exercise in commutative
algebra.  The statement about $\kk[Q]$ is the special case of
Theorem~\ref{t:monPrimary}, below, in which the binomial ideal~$I$ is
generated by monomials.
\end{proof}

What does a set of the form $(\alpha + \ZZ F) \cap Q$ look like?
Geometrically, it is roughly the (lattice points in the) intersection
of an affine subspace with a cone.  In the polynomial ring case, where
$Q = \NN^n$, a set $(\alpha + \ZZ F) \cap Q$ is always $\beta + F$ for
some lattice point $\beta \in \NN^n$: the intersection of a translate
of a coordinate subspace with the nonnegative orthant is a translated
orthant.  In fact, $F = \NN^J = \NN^n \cap \ZZ^J$ for some subset $J
\subseteq \{1,\ldots,n\}$, and then $\beta$ is obtained from $\alpha$
by setting all coordinates from~$J$ to~$0$ (if $\alpha$ has negative
coordinates outside of~$J$, then $\alpha + \ZZ F$ fails to
meet~$\NN^n$).  For general~$Q$, on the other hand, $(\alpha + \ZZ F)
\cap Q$ need not be a translate of~$F$.

\begin{example}\label{e:primary}
The prime ideal in Example~\ref{e:face} corresponds to the face~$F$
consisting of the nonnegative integer combinations of $e_1 + e_3$
and~$e_3$, where $e_1,e_2,e_3$ are the standard basis of~$\ZZ^3$.  In
terms of the depiction in Example~\ref{e:hartcone}, these are the
lattice points in~$Q$ that lie in the $xz$-plane.  The subgroup $\ZZ F
\subseteq \ZZ^3$ comprises all lattice points in the $xz$-plane.  Now
suppose that $\alpha = e_2$, the first lattice point along the
$y$-axis.  Then
$$%
\psfrag{a}{$\scriptstyle a$}
\psfrag{b}{$\scriptstyle b$}
\psfrag{c}{$\scriptstyle c$}
\psfrag{d}{$\scriptstyle d$}
\psfrag{x}{$\scriptstyle x$}
\psfrag{y}{$\scriptstyle y$}
\psfrag{z}{$\scriptstyle z$}
\begin{array}{cc@{}}
  \begin{array}{c@{}}
  \\
  \Bigg(
  \left[
  \begin{array}{c}
	 0\\
	 1\\
	 0
  \end{array}
  \right]
  + \ZZ\cdot\{xz\}\text{-plane}
  \Bigg) \cap Q
  \quad\ \longleftrightarrow
  \end{array}
&
  \begin{array}{@{}c@{}}\includegraphics{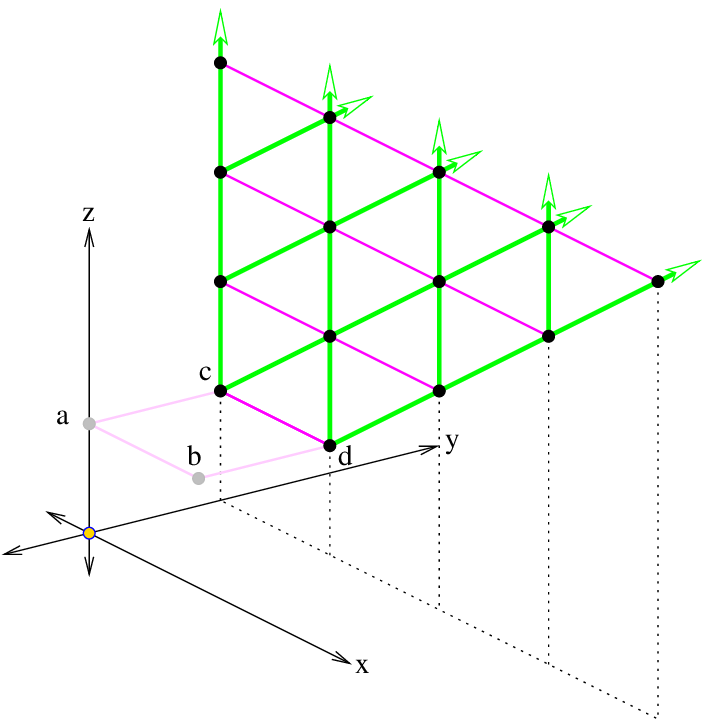}\end{array}
\end{array}
$$
is a union of two translates of~$F$.  For reference, the square over
which $Q$ is the cone is drawn lightly, while dotted lines fill out
part of the vertical plane $\alpha + \ZZ F$.
\end{example}

\begin{prop}
Every monomial ideal $I \subseteq \kk[Q]$ in an affine semigroup ring
has a unique minimal primary decompostion $I = P_1 \cap \cdots \cap
P_r$ as an intersection of monomial primary ideals~$P_i$ with distinct
associated primes.
\end{prop}
\begin{proof}[Proof sketch]
$I$ has a unique irredundant decompostion $I = W_1 \cap \cdots \cap
W_s$ as an intersection of \emph{irreducible monomial ideals}~$W_j$.
The existence of an irredundant irreducible decompostion can be proved
the same way irreducible decompostions are produced for arbitrary
submodules of noetherian modules.  The uniqueness of such a
decompostion, on the other hand, is special to monomial ideals in
affine semigroup rings \cite[Corollary~11.5]{cca}; it follows from the
uniqueness of \emph{irreducible resolutions}
\cite[Theorem~2.4]{irredres}.  See \cite[Chapter~11]{cca} for details.

Given the uniqueness properties of minimal monomial irreducible
decompostions, the (unique) monomial primary components are obtained
by intersecting all irreducible components sharing a given associated
prime.
\end{proof}

\begin{remark}
In polynomial rings, uniqueness of monomial irreducible decompostion
occurs for approximately the same reason that monomial ideals have
unique minimal monomial generating sets: the partial order on
irreducible ideals is particularly simple
\cite[Proposition~1.4]{rejecta}.  See \cite[\S5.2]{cca} for an
elementary derivation of existence and uniqueness of monomial
irreducible decompostion by Alexander duality.
\end{remark}

\subsection{Congruences on monoids}\label{s:cong}

The uniqueness of irreducible and primary decompostion of monomial
ideals rests, in large part, on the fine grading on~$\kk[Q]$, in which
the nonzero components $\kk[Q]_\alpha$ have dimension~$1$ as vector
spaces over~$\kk$.  Similar gradings are available for quotients
modulo binomial ideals, except that the gradings are by general
noetherian commutative monoids, rather than by free abelian groups or
by affine semigroups.  Our source for commutative monoids is Gilmer's
excellent book \cite{gilmer}.  For the special case of affine
semigroups, by which we mean finitely generated submonoids of free
abelian groups, see \cite[Chapter~7]{cca}.

For motivation, recall from Lemma~\ref{l:fiber} that the fibers of a
monoid morphism from $\NN^n$ to~$\ZZ^d$ have nice structure, and that
the polynomial ring $\kk[\xx]$ becomes graded by~$\ZZ^d$ via such a
morphism.  The fibers are the equivalence classes in an equivalence
relation, as is the case for any map $\pi: Q \to Q'$ of sets; but when
$\pi$ is a morphism of monoids, the equivalence relation satisfies an
extra condition.

\begin{defn}\label{d:congruence}
A \emph{congruence} on a commutative monoid~$Q$ is an equivalence
relation $\til$ that is \emph{additively closed}, in the sense that
$$%
  u \,\sim\, v\ \implies\ u\!+\!w \,\sim\, v\!+\!w
  \quad\text{for all } w \in Q.
$$
The \emph{quotient} $Q/\til$ is the set of equivalence classes
under addition.
\end{defn}

\begin{lemma}
The quotient $\oQ = Q/\til$ of a monoid by a congruence is a monoid.
Any congruence $\til$ on~$Q$ induces a $\oQ$-grading on the
\emph{monoid algebra} $\kk[Q] = \bigoplus_{u \in Q} \kk
\cdot\nolinebreak \ttt^u$ in which the \emph{monomial} $\ttt^u$ has
\emph{degree}~$\ol u \in \oQ$ whenever $u \mapsto \ol u$.
\end{lemma}
\begin{proof}
This is an easy exercise.  It uses that the multiplication on~$\kk[Q]$
is given by $\ttt^u \ttt^v = \ttt^{u+v}$ for $u,v \in Q$.
\end{proof}

\begin{defn}\label{d:induced}
In any monoid algebra $\kk[Q]$, a binomial ideal $I \subseteq \kk[Q]$
generated by binomials $\ttt^u - \lambda\ttt^v$ with $\lambda \in \kk$
\emph{induces} a congruence~$\til$ (often denoted by~$\til_I$)~in~which
$$%
  u \sim v \text{ if } \ttt^u - \lambda\ttt^v \in I \text{ for some }
  \lambda \neq 0.
$$
\end{defn}

\begin{lemma}\label{l:hilb}
Fix a binomial ideal $I \subseteq \kk[Q]$ in a monoid algebra.  Then
$I$ and\/ $\kk[Q]/I$ are both graded by $\oQ = Q/\til$.  The
\emph{Hilbert function} $\oQ \mapsto \NN$, which for any $\oQ$-graded
vector space~$M$ takes $\ol q \mapsto \dim_\kk\,M_\ol q$, satisfies
$$%
  \dim_\kk\,(\kk[Q]/I)_\ol q\ =\
  \begin{cases}
    0 &\text{if \,} \ol q = \{u \in Q \mid \ttt^u \in I\}
  \\1 &\text{otherwise.}
  \end{cases}
$$
\end{lemma}

The proof of the lemma is another simple exercise.  To rephrase, it
says that every pair of monomials in a given congruence class
under~$\til_I$ are equivalent up to a nonzero scalar modulo~$I$, and
the only monomials mapping to~$0$ are in~$I$.  A slightly less
set-theoretic and more combinatorial way to think about congruences
uses graphs.

\begin{defn}\label{d:graph}
Any binomial ideal $I \subseteq \kk[Q]$ defines a graph~$G_I$ whose
vertices are the elements of the monoid~$Q$ and whose (undirected)
edges are the pairs $(u,v) \in Q \times Q$ such that $\ttt^u -
\lambda\ttt^v \in I$ for some nonzero $\lambda \in \kk$.  Write $\pi_0
G_I$ for the set of connected components of~$G_I$.
\end{defn}

Thus $C \in \pi_0 G_I$ is the same thing as a congruence class
under~$\til_I$.  The moral of the story is that combinatorics of the
graph~$G_I$ controls the (binomial) primary decompostion of~$I$.

\begin{example}\label{e:graph}
Each of the two binomial generators of the ideal
$$%
\psfrag{x}{$\scriptstyle x$}
\psfrag{y}{$\scriptstyle y$}
\begin{array}{cc}
  \begin{array}{c}
  I = \<{\color{darkpurple}x^2-xy}, {\color{darkgreen}xy-y^2}\>
  \subseteq \kk[x,y]
  \qquad\Implies\ \
  \end{array}
&
  G_I\ \ = \begin{array}{c}\includegraphics{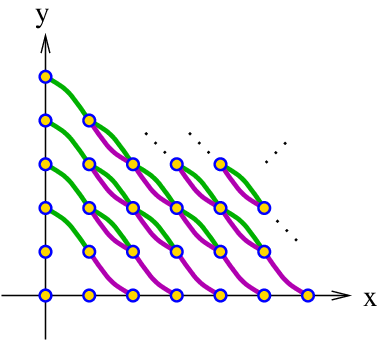}\end{array}
\end{array}
$$%
determines a collection of edges of the graph~$G_I$, indicated in the
figure, by additivity of the congruence~$\til_I$.  In reality, $G_I$
has many more edges than those depicted: since~$\til_I$ is an
equivalence relation, every connected component is a complete graph on
its vertex set.  However, in examples, it is convenient to draw---and
more helpful to see---only edges determined by monomial multiples of
generating binomials.

The connected components of~$G_I$ are the fibers of the monoid
morphism from $\NN^2$~to
$$%
  \NN^2/\til_I\ =\ \begin{array}{c}\includegraphics{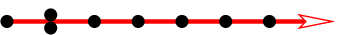}\end{array}
$$
the monoid~$\NN$ with the element~$1$ ``doubled''.  The ideal $I$ has
primary decomposition
$$%
  I = \<x^2, xy, y^2\> \cap \<x - y\>.
$$
The first primary component reflects the three singleton components
of~$G_I$ near the origin.  The other primary component reflects the
diagonal connected components of~$G_I$ marching off to infinity.

Although the $\NN$-graded Hilbert function of $\kk[x,y]/I$ takes the
values $1,2,1,1,1,\ldots$, the $\ol\NN{}^2$-graded Hilbert function
takes only the value~$1$.
\end{example}

\begin{example}
When $I = I_A$ is the toric ideal for a matrix~$A$, the connected
components of~$G_I$ are the fibers~$F_\alpha$ for $\alpha \in \NN A$.
\end{example}

\begin{example}\label{e:rhoJ}
If $I = I_{\rho,J}$ is a binomial prime in a polynomial
ring~$\kk[\xx]$ over an algebraically closed field~$\kk$, and $C \in
\pi_0 G_I$ is a connected component, then either $C = \NN^n \minus
\NN^J$ or else $C = (u + L) \cap \NN^J$ for some $u \in \NN^J$.  When
$\rho$ is the trivial (only) character on the
lattice $L = \{0\} \subseteq \ZZ^{\{3\}}$ and $J = \{3\} \subseteq
\{1,2,3\}$, for instance, then
$$%
\psfrag{x}{$\scriptstyle x$}
\psfrag{y}{$\scriptstyle y$}
\psfrag{z}{$\scriptstyle z$}
\begin{array}{cc}
  \begin{array}{c}
  I = I_{\rho,J} = \<x,y\>
  \subseteq \kk[x,y,z]
  \qquad\Implies\ \
  \end{array}
&
  G_I\ \ = \begin{array}{c}\includegraphics{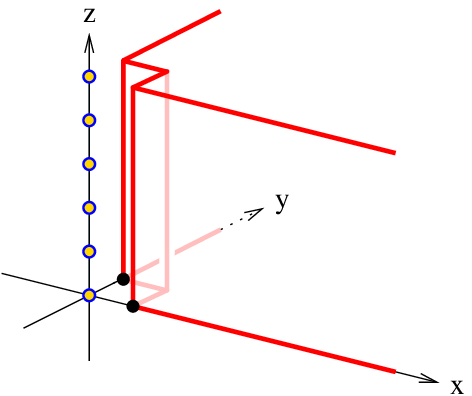}\end{array}
\end{array}
$$%
In this case $\NN^n \minus \NN^J$ consists of the monomials off of the
vertical axis (i.e., those in the region
outlined by bold straight lines), whereas every component $(u + L)
\cap \NN^J$ of~$G_I$ is simply a single lattice point on the vertical
axis.
\end{example}

\subsection{Binomial primary ideals with monomial associated primes}\label{s:binomPrim}

The algebraic characterization of monomial primary ideals in the first
half of Proposition~\ref{p:monomPrimary} has an approximate analogue
for primary binomial ideals, although it requires hypotheses on the
base field~$\kk$.

\begin{thm}\label{t:algbinom}
Fix\/ $\kk$ algebraically closed of characteristic~$0$.  If $I
\subseteq \kk[\xx]$ is an $I_{\rho,J}$-primary binomial ideal, then $I
= I_\rho + B$ for some binomial ideal~$B \supseteq \mm_J^\ell$
with~$\ell > 0$.
\end{thm}
\begin{proof}
This is the characterization of primary decomposition
\cite[Theorem~7.1]{binomialIdeals} applied to a binomial ideal that is
already primary.
\end{proof}

The content of the theorem is that $I$ contains both $I_\rho$ and a
power of each variable $x_i$ for $i \notin J$.  (In positive
characteristic, $I$ contains a Frobenius power of~$I_\rho$, but not
necessarily $I_\rho$ itself.)  A more precise analogue of the
algebraic part of Proposition~\ref{p:monomPrimary}
would characterize which binomial ideals~$B$ result in primary ideals.
However, there is no simple way to describe such binomial ideals~$B$
in terms of generators.
The best that can be hoped for is an answer to a pair of questions:
\begin{enumerate}
\item%
What analogue of standard monomials allows us to ascertain when $I$ is
primary?
\item%
What property of the standard monomials characterizes the primary
condition?
\end{enumerate}
Preferably the answers should be geometric, and suitable for binomial
ideals in arbitrary affine semigroup rings, as in the second half of
Proposition~\ref{p:monomPrimary}.

Lemma~\ref{l:hilb} answers the first question.  Indeed, if $I
\subseteq \kk[Q]$ is a monomial ideal, then $\std(I)$ is exactly the
subset of~$Q$ such that $\kk[Q]/I$ has $Q$-graded Hilbert function
$\dim_\kk\,(\kk[Q]/I)_u = 1$ for $u \in \std(I)$ and $0$ for $\ttt^u
\in I$.  In the monomial case, we could still define the monoid
quotient $\oQ = Q/\til_I$, whose classes are all singleton monomials
except for the class of monomials in~$I$.  Therefore, for general
binomial ideals~$I$, the set of non-monomial classes of the
congruence~$\til_I$ plays the role~of~$\std(I)$.

Note that Proposition~\ref{p:monomPrimary} answers the second question
for monomial ideals in affine semigroup rings, whose associated primes
are automatically monomial.  The next step relaxes the condition
on~$I$ but not on the associated prime:
consider a binomial ideal $I$ in an affine semigroup ring~$\kk[Q]$,
and ask when it is $\pp_F$-primary for a
face $F \subseteq Q$.  The answer in this case relies, as promised, on
the combinatorics of the graph~$G_I$ from Definition~\ref{d:graph} and
its set $\pi_0 G_I$ of connected components; however, as
Proposition~\ref{p:monomPrimary} hints, the group generated by~$F$
enters in an essential way.

\begin{defn}\label{d:localize}
For a face $F$ of an affine semigroup~$Q$, and any
$\kk[Q]$-module~$M$,
$$%
  M[\ZZ F] = M \otimes_{\kk[Q]} \kk[Q + \ZZ F]
$$
is the localization
by inverting all monomials not in~$\pp_F$.  If $I \subseteq \kk[Q]$ is
a binomial ideal, then a connected component $C \in \pi_0 G_I$ is
\emph{$F$-finite} if $C = C' \cap Q$ for some finite connected
component~$C'$ of the graph~$G_{I[\ZZ F]}$ for the localization $I[\ZZ
F] \subseteq \kk[Q][\ZZ F]$.
\end{defn}

Thus, for example, $\kk[Q][\ZZ F] = \kk[Q + \ZZ F]$ is the monoid
algebra for the affine semigroup $Q + \ZZ F$ obtained by inverting the
elements of~$F$ in~$Q$.  In Proposition~\ref{p:monomPrimary}, where $I
\subseteq \kk[Q]$ is a monomial ideal, all of the connected components
of~$G_I$ inside of~$\std(I)$ are singletons, and the same is true of
$G_{I[\ZZ F]}$.  See Section~\ref{s:minMonom} for additional
information and examples concerning the geometry and combinatorics of
localization.

\begin{example}\label{e:finite}
Connected components of~$G_I$ can be finite but not $F$-finite for a
given face~$F$ of~$Q$.  For instance, if $I = \<x - y\> \subseteq
\kk[x,y] = \kk[\NN^2]$, then the connected components of~$G_I$ are all
finite---they correspond to the sets of monomials in~$\kk[x,y]$ of
fixed total degree---but not $F$-finite if $F$ is the horizontal axis
of $Q = \NN^2$: once $x$ is inverted, $\pi_0 G_{I[x^{-1}]}$ consists
of infinite northwest-pointing rays in the upper half-plane.
\end{example}

\begin{thm}\label{t:monPrimary}
Fix a monomial prime ideal $\pp_F = \<\ttt^u \mid u \notin F\>$ in an
affine semigroup ring~$\kk[Q]$ for a face $F \subseteq Q$.  A binomial
ideal $I \subseteq \kk[Q]$ is $\pp_F$-primary if and only if
\begin{enumerate}
\item%
Every connected component of\/~$G_I$ other than $\{u \in Q \mid \ttt^u
\in I\}$ is $F$-finite.
\item%
$F$ acts on the set of $F$-finite components semifreely with finitely
many orbits.
\end{enumerate}
\end{thm}

Thus the primary condition is fundamentally a finiteness
condition---or really a pair of finiteness conditions.  A proof is
sketched after Example~\ref{e:binomPrimary};
but first, the terminology requires precise explanations.
Semifreeness, for example,
guarantees that the set of $F$-finite components is a subset of a set
acted on freely by $\ZZ F$;
this is part of the characterization of semifree actions in
\cite{mesoprimary}.

\begin{defn}
An \emph{action} of a monoid~$F$ on a set~$T$ is a map $F \times T \to
T$, written $(f,t) \mapsto f + t$, that satisfies $0 + t = t$ for all
$t \in T$ and respects addition: $(f + g)+ t = f + (g + t)$.  The
monoid action is \emph{semifree} if $t \mapsto f + t$ is an injection
$T \into T$ for each $f \in F$, and $f \mapsto f + t$ is an injection
$F \into T$ for each $t \in T$.
\end{defn}

In contrast to group actions, monoid actions do not a~priori define
equivalence relations, because the relation $t \sim f + t$ can fail to
be symmetric.
The relation is already reflexive
and transitive, however, precisely by the two axioms for
monoid~actions.

\begin{defn}\label{d:orbit}
An \emph{orbit} of a monoid action of~$F$ on~$T$ is an equivalence
class under the symmetrization of the relation $\{(s,t) \mid f + s
= t$ for some $f \in F\} \subseteq T \times T$.
\end{defn}

Combinatorially, if $F$ acts on~$T$, one can construct a directed
graph
with vertex set $T$ and an edge from $s$ to~$t$ if $t = f + s$ for
some $f \in F$.
Then an orbit is a connected component of the underlying undirected
graph.

\begin{example}\label{e:binomPrimary}
The ideal
$$%
\psfrag{x}{$\scriptstyle x$}
\psfrag{y}{$\scriptstyle y$}
\psfrag{z}{$\scriptstyle z$}
\begin{array}{cc}
\\[-3ex]
  \begin{array}{c}
  I = \<{\color{darkpurple}x-y},x^2\>
  \subseteq \kk[x,y,z]
  \qquad\Implies\ \
  \end{array}
&
  G_I\ \ = \begin{array}{c}\includegraphics{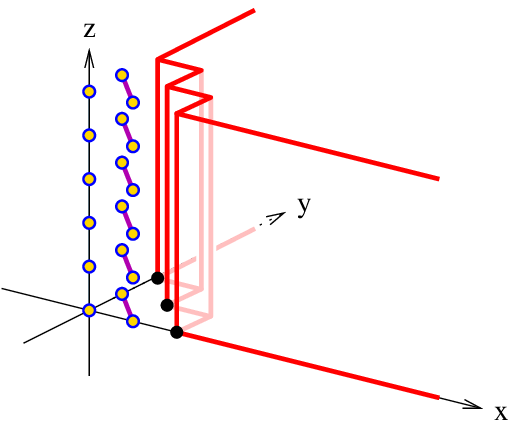}\end{array}
\end{array}
$$%
is $\pp$-primary for $\pp = \pp_F = \<x,y\>$, where $F$ is the
$z$-axis of~$\NN^3$.  The monoid~$F$ acts on the $F$-finite connected
components of~$G_I$ with two orbits: one on the $z$-axis, where each
connected component is a singleton; and one adjacent orbit, where
every connected component is a pair.  The monomial class in this
example (outlined by bold straight lines) is the set of monomials
in~$\<x^2,xy,y^2\>$.
\end{example}

\begin{example}
Fix notation as in Examples~\ref{e:face} and~\ref{e:primary}.  The
ideal $I = \<c^2,cd,d^2\> \subseteq \kk[Q]$ is $\pp_F$-primary.  In
contrast to Example~\ref{e:binomPrimary}, this time the $F$-finite
connected components are all singletons, but the two orbits are not
isomorphic as sets acted on by the face~$F$: one orbit is $F$ itself,
while the other is the set depicted in Example~\ref{e:primary}.
\end{example}

\begin{proof}[Proof sketch for Theorem~\ref{t:monPrimary}]
This theorem is the core conclusion of \cite[Theorem~2.15 and
Proposition~2.13]{primDecomp}.
The argument is summarized as follows.

For any set~$T$, let $\kk\{T\}$ denote the vector space over~$\kk$
with basis~$T$.  If $\pi_0 G_I$ satisfies the two conditions, then
$\kk[Q]/I$ has finite a filtration, as a $\kk[Q]$-module, whose
associated graded pieces are the vector spaces $\kk\{T\}$ for the
finitely many $F$-orbits~$T$ of $F$-finite components of~$G_I$.  In
fact,
semifreeness guarantees that for each orbit~$T$, the vector space
$\kk\{T\}$ is naturally a torsion-free module over $\kk[F] =
\kk[Q]/\pp_F$.  Finiteness of the number of orbits guarantees that the
associated graded module of~$\kk[Q]/I$ is a finite direct sum of
modules~$\kk\{T\}$, so it has only one associated prime,
namely~$\pp_F$.  Consequently $\kk[Q]/I$ itself has just one
associated~prime.

For the other direction, when $I$ is $\pp_F$-primary, one proves that
inverting the monomials and binomials outside of~$\pp_F$ annihilates
the $\oQ$-graded pieces of $\kk[Q]/I$ for which the connected
component in~$G_{I[\ZZ F]}$ is infinite \cite[Lemmas~2.9
and~2.10]{primDecomp}.  Since the elements outside of~$\pp_F$ act
injectively on $\kk[Q]/I$ by definition of $\pp_F$-primary,
\mbox{every} class of~$\til_I$ that is not $F$-finite must therefore
already consist of monomials in~$I$.  The semifree action of~$F$ on
the $F$-finite components derives simply from the fact that $\kk[Q]/I$
is torsion-free as a $\kk[F]$-module, where the $\kk[F]$-action is
induced by the inclusion $\kk[F] \subseteq \kk[Q]$.  Generalities
about $\oQ$-gradings of this sort imply that $\kk[Q]/I$ possesses a
filtration whose associated graded pieces are as in the previous
paragraph.  The minimality of~$\pp_F$ over~$I$ implies that the length
of the filtration is finite.
\end{proof}

\section{Binomial primary decomposition}\label{s:decomp}

General binomial ideals induce more complicated congruences than
primary binomial ideals.  This section completes the combinatorial
analysis of binomial primary
decomposition by describing how to pass from an arbitrary binomial
ideal to its primary components.
There are crucial points where characteristic zero or algebraically
closed hypotheses are required of the field~$\kk$, but those will be
mentioned explicitly; if no mention is made, then $\kk$ is assumed to
be arbitrary.

\subsection{Monomial primes minimal over binomial ideals}\label{s:minMonom}

The first step is to consider again the setting from the previous
section, particularly Theorem~\ref{t:monPrimary}, where a monomial
prime ideal~$\pp_F$ is associated to~$I$ in an arbitrary affine
semigroup ring~$\kk[Q]$, except that now the binomial ideal~$I$ is not
assumed to be primary.  The point is to construct its $\pp_F$-primary
component.  The nature of Theorem~\ref{t:monPrimary} splits the
construction into two parts:
\begin{enumerate}
\item%
ensuring that the (non-monomial) connected components are all
$F$-finite, and
\item%
forcing the face $F$ to act in the correct manner on the components.
\end{enumerate}
These operations will be carried out in reverse order, with part~2
being accomplished by localization, and then part~1 being accomplished
by simply lumping all of the
connected components that are not $F$-finite together.

\begin{defn}
For a face $F$ of an affine semigroup~$Q$ and a binomial ideal $I
\subseteq\nolinebreak \kk[Q]$,
$$%
  (I:\ttt^F)
  = I[\ZZ F] \cap \kk[Q]
$$
is the kernel of the composite map $\kk[Q] \to \kk[Q]/I \to
(\kk[Q]/I)[\ZZ F]$.
\end{defn}

\begin{remark}
The notation $(I:\ttt^F)$ is explained by an equivalent construction
of this ideal.  Indeed, the usual meaning of the colon operation for
an element $y \in \kk[Q]$ is that $(I:y) = \{z \in \kk[Q] \mid yz \in
I\}$.  Here, $(I:\ttt^F) = (I:\ttt^f)$ for any lattice point $f$ lying
sufficiently far in the relative interior of~$F$.  Equivalently,
$(I:\ttt^F) = (I:(\ttt^f)^\infty) = \bigcup_{r \in \NN} (I:\ttt^{rf})$
for any lattice point $f \in F$ that does not lie on a proper~subface
of~$F$.
\end{remark}

Combinatorially, the passage from $I$ to $(I:\ttt^F)$ has a concrete
effect.

\begin{lemma}\label{l:colon}
The connected components of the graph $G_{(I:\ttt^F)}$ defined
by~$(I:\ttt^F)$ are obtained from $\pi_0 G_I$ by joining together all
pairs of components $C_u$ and~$C_v$ such that $C_{u+f} = C_{v+f}$ for
some $f \in F$, where $C_u$ is the connected component containing $u
\in Q$.
\end{lemma}
\begin{proof}
Two lattice points $u,v \in Q$ lie in the same component
of~$G_{(I:\ttt^F)}$ precisely when there is a binomial $\ttt^u -
\lambda\ttt^v \in \kk[Q]$ such that $\ttt^f(\ttt^u - \lambda\ttt^v)
\in I$ for some $f \in F$.
\end{proof}

Roughly speaking: join the components if they become joined after
moving them up by an element $f \in F$.  Illustrations of lattice
point phenomena related to binomial primary decomposition become
increasingly difficult to draw in two dimensions as the full nature of
the theory develops, but a small example is possible at this stage.

\begin{example}\label{e:localize}
The ideal $I = \<xz-yz,x^2-x^3\> \subseteq \kk[x,y,z]$ yields the same
graph as Example~\ref{e:monomPrimary'} except for two important
differences:
\begin{itemize}
\item%
here there are many fewer edges in the $xy$-plane ($z = 0$); and
\item%
every horizontal ($z = \text{constant} \geq 1$) slice of the big
region is a separate connected component, in contrast to
Example~\ref{e:monomPrimary'}, where the entire big region
was a single connected component corresponding to the monomials in
the~ideal.
\end{itemize}
$$%
\psfrag{x}{$\scriptstyle x$}
\psfrag{y}{$\scriptstyle y$}
\psfrag{z}{$\scriptstyle z$}
\begin{array}{@{}c@{\ }c@{}}
\\[-3ex]
  \begin{array}{@{}c}
  I = \<{\color{darkpurple}xz-yz},{\color{darkgreen}x^2-x^3}\>
  \subseteq \kk[x,y,z]
  \quad\Implies\ 
  \end{array}
&
  G_I\ \ \approx \begin{array}{@{\!}c@{}}\includegraphics{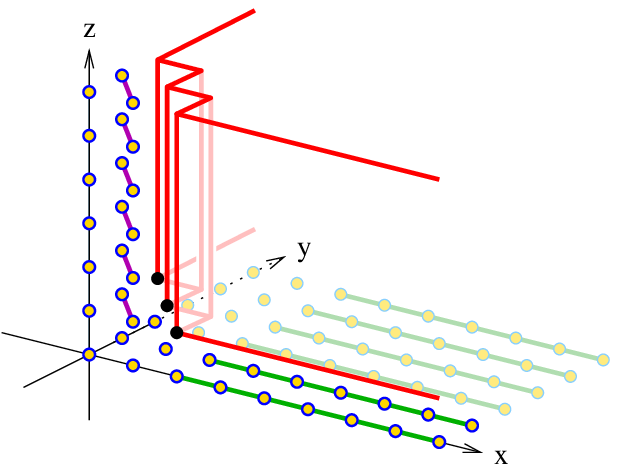}\end{array}
\\[-1ex]
\end{array}
$$%
Only the second generator, ${\color{darkgreen}x^2-x^3}$, is capable of
joining pairs of points in the $xy$-plane, and it does so parallel to
the $x$-axis, starting at $x = 2$.  Of course,
${\color{darkgreen}x^2-x^3}$ also joins pairs of points in the same
manner at positive heights $z \geq 1$, but again only starting at $x =
2$.  The first generator, ${\color{darkpurple}xz-yz}$, has the same
effect as ${\color{darkpurple}x-y}$ did in
Example~\ref{e:monomPrimary'}, except that ${\color{darkpurple}xz-yz}$
only joins pairs of lattice points at height $z = 1$ or more.  In
summary, every connected component of~$G_I$ in this example is
contained in a single horizontal slice, and the horizontal slices
of~$G_I$ are
$$%
\psfrag{x}{$\scriptstyle x$}
\psfrag{y}{$\scriptstyle y$}
\psfrag{z=0}{$z = 0$}
\psfrag{z>0}{$z \geq 1.$}
\psfrag{and}{and}
\begin{array}{c}\\[-3.5ex]\includegraphics{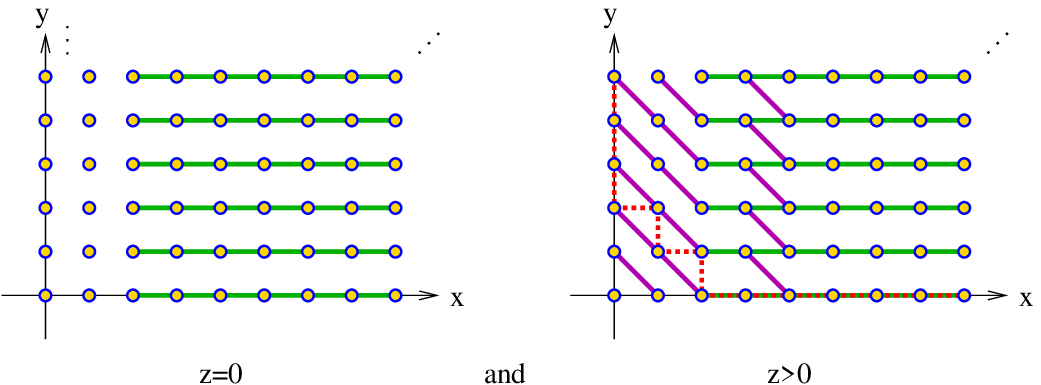}\\[-.5ex]\end{array}
$$
The outline of the big region
is drawn as a dotted line in the $z \geq 1$ slice illustration, which
depicts only enough of the edges to elucidate its three connected
components.

Let $F$ be the part of~$\NN^3$ in the $xy$-plane, so $\pp_F = \<z\>$.
The ideal $(I:\ttt^F) = (I:z^\infty) = (I:z) = \<x-y,x^2-x^3\>$ again
has the property that every connected component of~$G_{(I:z)}$ is
contained in a single horizontal slice, but now all of these slices
look like the $z \geq 1$ slices of~$G_I$.  Compare this to the
statement of Lemma~\ref{l:colon}.
\end{example}

\begin{thm}\label{t:monAss}
Fix a monomial prime $\pp_F = \<\ttt^u \mid u \notin F\>$ in an affine
semigroup ring $\kk[Q]$ for a face $F \subseteq Q$.  If\/~$\pp_F$ is
minimal over a binomial ideal $I \subseteq \kk[Q]$ and $\ol\pi_0
G_{I[\ZZ F]}$~is the set of finite
components of the graph $G_{I[\ZZ F]}$, then the
$\pp_F$-primary~component~of~$I$~is
$$%
  (I:\ttt^F) + \<\ttt^u \mid C_u \notin \ol\pi_0 G_{I[\ZZ F]}\>.
$$
The exponents on the monomials in this primary component are precisely
the elements of~$Q$ that lie in infinite connected components of the
graph~$G_{I[\ZZ F]}$.
\end{thm}
\begin{proof}
This is \cite[Theorem~2.15]{primDecomp}.  The $\pp_F$-primary
component of~$I$ is equal to the $\pp_F$-primary component of
$(I:\ttt^F)$ because primary decomposition is preserved by
localization (see
\cite[Proposition~4.9]{AM}, for example), so we may as well assume
that $I = (I:\ttt^F)$.
It is elementary to check that $F$ acts on the connected components.
The action is semifree on the $F$-finite components, for if $f + u
\sim g + u$ for some $f,g \in F$, then $u \sim n(f-g) + u$ for all $n
\in \NN$, whence $C_u \in \pi_0 G_{I[\ZZ F]}$ is infinite; and if $f +
u \sim f + v$ then $u \sim v$, because $\xx^f$ is a unit on $I[\ZZ
F]$.  It is also elementary, though nontrivial, to check that the
kernel of the usual localization homomorphism $\kk[Q]/I \to
\kk[Q]_{\pp_F}/I_{\pp_F}$---inverting all polynomials outside
of~$\pp_F$, not just monomials---contains every monomial $\ttt^u$ for
which $C_u \notin \ol\pi_0 G_{I[\ZZ F]}$ \cite[Lemmas~2.9
and~2.10]{primDecomp}.  Now note that $(I:\ttt^F) + \<\ttt^u \mid C_u
\notin \ol\pi_0 G_{I[\ZZ F]}\>$ is already
primary by~Theorem~\ref{t:monPrimary}.
\end{proof}

\begin{remark}
The graph of~$I[\ZZ F]$ has vertex set $Q[\ZZ F]$, which naturally
contains~$Q$.  That is why, in Definition~\ref{d:localize} and
Theorem~\ref{t:monAss}, it makes sense to say that an element of~$Q$
lies in a connected component of~$G_{I[\ZZ F]}$.
\end{remark}

\begin{example}
The $\pp_F$-primary component of the ideal $I$ in
Example~\ref{e:localize} is the ideal~$I$ in
Example~\ref{e:monomPrimary'}: the differences that remain, after the
localization operation in Example~\ref{e:localize} is complete, are
erased by lumping together the infinite connected components into a
single monomial component.
\end{example}

\begin{example}\label{e:finite'}
Starting with $(I:\ttt^F)$ in Theorem~\ref{t:monAss}, it is not enough
to throw in the monomials whose exponents lie in infinite components
of~$G_{(I:\ttt^F)}$; that is, a connected component
of~$G_{(I:\ttt^F)}$ could be finite but nonetheless equal to the
intersection with~$Q$ of an infinite component of~$G_{I[\ZZ F]}$.
This occurs for $I = \<xz - yz\> \subseteq \kk[x,y,z]$, with $F$ being
the $xy$-coordinate plane of $Q = \NN^3$, so $\pp_F = \<z\>$.  Every
connected component of~$G_I$ is finite, even though $I = (I:\ttt^F)$
is not primary.  When $x$ and~$y$ are inverted to form~$I[\ZZ F]$, the
components at height $z \geq 1$ become cosets of the line spanned by
$[\twoline {\phantom-1}{-1}]$, whose intersections with $\NN^3$ are
bounded.  Hence the $\<z\>$-primary component of~$I$ is $I + \<z\> =
\<z\>$, as is clear from the primary decomposition $I = \<z\> \cap
\<x-y\>$.
\end{example}

\subsection{Primary components for arbitrary given associated primes}\label{s:given}

For this subsection, fix a binomial ideal $I \subseteq \kk[\xx]$ in a
polynomial ring with a binomial associated prime $I_{\rho,J}$ for some
character $\rho: L \to \kk^*$ defined on a saturated sublattice $L
\subseteq \ZZ^J$.
Now it is important to assume that the field $\kk$ is algebraically
closed of characteristic~$0$, for these hypotheses are crucial to the
truth of Theorem~\ref{t:algbinom}, and that theorem is the tool that
reduces the current general situation to the special case in
Section~\ref{s:minMonom}.  The logic is as follows.

Every binomial $I_{\rho,J}$-primary ideal contains~$I_\rho$ by
Theorem~\ref{t:algbinom}.  Since
we are trying to construct a binomial $I_{\rho,J}$-primary component
of~$I$ starting from $I$ itself, the first step should therefore be to
enlarge~$I$ by throwing in~$I_\rho$.  Here is a formal statement.

\begin{prop}\label{p:polyringPrimary}
Fix a binomial ideal $I \subseteq \kk[\xx]$
with $\kk$ algebraically closed of characteristic~$0$.  If $P$ is any
binomial $I_{\rho,J}$-primary component of~$I$, then $P$ is the
preimage in $\kk[\xx]$ of a binomial $(I_{\rho,J}/I_\rho)$-primary
component of $(I + I_\rho)/I_\rho \subseteq \kk[\xx]/I_\rho$.
\end{prop}

An alternative phrasing makes the point of considering the quotient
$\kk[\xx]/I_\rho$ clearer.

\begin{prop}\label{p:polyringPrimary'}
If $P$ is an $I_{\rho,J}$-primary binomial ideal in~$\kk[\xx]$, with
$\kk$ algebraically closed of characteristic~$0$, then the image
of~$P$ in the affine semigroup ring $\kk[Q] = \kk[\xx]/I_\rho$ is a
binomial ideal $\pp_F$-primary to the monomial prime $\pp_F =
I_{\rho,J}/I_\rho$ in~$\kk[Q]$.
\end{prop}
\begin{proof}
This is an immediate consequence of Theorem~\ref{t:algbinom} and
Theorem~\ref{t:prime} along with Corollary~\ref{c:prime}: the affine
semigroup~$Q$ is $(\NN^J\!/L) \times \NN^\oJ$, and the face~$F$ is the
copy of~$\NN^J\!/L = (\NN^J\!/L) \times \{0\}$ in~$Q$.
\end{proof}

Thus the algebra of
general binomial associated primes for polynomial rings is lifted from
the algebra of monomial associated primes in affine semigroup rings.
The final step is isolating how the combinatorics, namely
Theorem~\ref{t:monPrimary}, lifts.  Since the algebra of quotienting
$\kk[\xx]$ modulo~$I_\rho$ corresponds to the quotient of~$\NN^n$
modulo~$L$, we expect the lifted finiteness conditions to involve
cosets of~$L$.

\begin{defn}
A subset of~$\NN^n$ is \emph{$L$-bounded} for a sublattice $L
\subseteq \ZZ^n$ if the subset is contained in a finite union of
cosets of~$L$.
\end{defn}

\begin{cor}\label{c:primDecomp}
Fix a binomial ideal $I \subseteq \kk[\xx]$ with $\kk$ algebraically
closed of characteristic~$0$.  If $I_{\rho,J}$ is minimal over~$I$,
then the $I_{\rho,J}$-primary component of~$I$ is
$$%
  P = I' + \<\xx^u \mid C_u \in \pi_0 G_{I'[\ZZ^J]} \text{ is not $L$-bounded}\>,
$$
where $I'[\ZZ^J]$ is the localization along~$\NN^J$, and $I'$ is
defined, using $\xx_J = \prod_{j \in J} x_j$, to be
$$%
  I' = \big((I + I_\rho):\xx_J^\infty\big) = (I + I_\rho)[\ZZ^J] \cap \kk[\xx].
$$
If $I_{\rho,J}$ is associated to~$I$ but not minimal over~$I$, then
for any monomial ideal $K$ containing a sufficiently high power of\/
$\mm_J = \<x_i \mid i \notin J\>$, an $I_{\rho,J}$-primary component
of~$I$ is defined as $P$ is, above, but using $I'_K$ in place of~$I'$,
where
$$%
  I'_K = \big((I + I_\rho + K):\xx_J^\infty\big).
$$
\end{cor}
\begin{proof}[Proof sketch]
This is \cite[Theorem~3.2]{primDecomp}.  The key is to lift the
monomial minimal prime case for affine semigroup rings in
Theorem~\ref{t:monAss} to the current binomial associated prime case
in polynomial rings using Propositions~\ref{p:polyringPrimary}
and~\ref{p:polyringPrimary'}.  For
an embedded prime~$I_{\rho,J}$, one notes that any given
$I_{\rho,J}$-primary component of~$I$ must contain a sufficiently high
power of~$\mm_J$, so it is logical to begin the search for an
$I_{\rho,J}$-primary component by simply throwing such monomials along
with~$I_\rho$ into~$I$.
But then $I_{\rho,J}$ is minimal over the resulting ideal $I + I_\rho
+ K$, so the minimal prime case applies.
\end{proof}

The definition of~$P$ in the theorem says that $G_P$ has two types of
connected components: the ones that are $L$-bounded upon localization
along~$\ZZ^J$, and the connected component consisting of exponents on
monomials in~$P$.  The theorem says that $G_P$ shares all but its
monomial component with the graph~$G_{I'}$, and that the other
connected components of~$G_{I'}$ fail to remain $L$-bounded upon
localization along~$\ZZ^J$.  In the case where $I_{\rho,J}$ is an
embedded prime, the graph-theoretic explanation is that $G_{I'}$ has
too many connected components that remain $L$-bounded upon
localization; in fact, there are infinitely many $\NN^J$-orbits.  The
hack of adding~$K$ throws all but finitely many $\NN^J$-orbits into
the $L$-infinite ``big monomial'' connected component.

\begin{example}
A primary decomposition of the ideal $I =
\<{\color{darkpurple}x^2-xy}, {\color{darkgreen}xy-y^2}\>$ was already
given in Example~\ref{e:graph}.  Analyzing it from the perspective of
Corollary~\ref{c:primDecomp}
completes the heuristic insight.

First let $I_{\rho,J} = \<x - y\>$, so $J = \{1,2\}$ and $\rho: L \to
\kk^*$ is trivial on the lattice~$L$ generated by~$[\twoline
{\phantom-1}{-1}]$.
Then $I' = I + \<x - y\>$ is already prime.  Passing from $I$ to~$I'$
has the sole effect of joining the two isolated points (the basis
vectors) on the axes together.

Now let $I_{\rho,J} = \<x,y\>$, so $J = \nothing$ and $\rho$ is the
trivial (only) character defined on $L = \ZZ^J = \{0\}$.  Every
connected component of~$G_I$ in Example~\ref{e:graph} remains
$L$-bounded upon localization along~$\ZZ^J$, but there are infinitely
many such components.  Choosing $K = \<x,y\>^e$, so that $I'_K = I' +
\<x,y\>^e$, kills off all but finitely many, to~get
$$%
\psfrag{e}{${\color{red}\scriptstyle e}$}
\psfrag{x}{$\scriptstyle x$}
\psfrag{y}{$\scriptstyle y$}
\begin{array}{@{}cc@{}}
\\[-3ex]
  \begin{array}{@{}c@{}}
  P = \<{\color{darkpurple}x^2-xy}, {\color{darkgreen}xy-y^2}\> +
  {\color{red}\<x,y\>^e} \subseteq \kk[x,y] 
  \quad\;\Implies\
  \end{array}
&
  G_P\ \ = \begin{array}{c@{}}\includegraphics{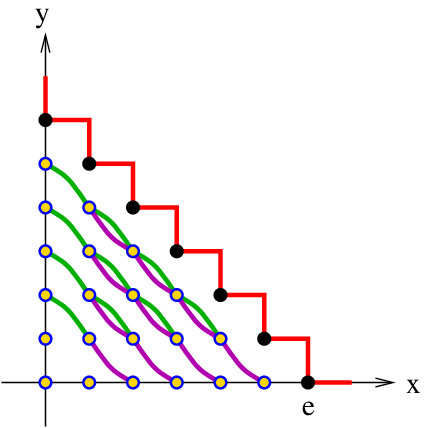}\end{array}
\end{array}
$$%
In particular, taking $e = 2$ recovers the primary decomposition from
Example~\ref{e:graph}.
\end{example}

\subsection{Finding associated primes combinatorially}\label{s:finding}

The constructions of binomial primary components in previous sections
assume that a monomial or binomial associated prime of a binomial
ideal has been given.  To conclude the discussion of primary
decomposition of binomial ideals, it remains to examine the set of
associated primes.  The existence of binomial primary decompositions
hinges on a fundamental result, due to Eisenbud and Sturmfels
\cite[Theorem~6.1]{binomialIdeals}, that was a starting point for all
investigations involving primary decomposition of
\mbox{binomial}~ideals.

\begin{thm}\label{t:es}
Every associated prime of a binomial ideal in~$\kk[\xx]$ is a binomial
prime if the field\/~$\kk$ is algebraically closed.
\end{thm}

Although the statement is for polynomial rings, a simple reduction
implies the existence of binomial primary decomposition in the
generality of monoid algebras as defined in Section~\ref{s:cong},
given the construction of binomial primary components.

\begin{cor}\label{c:char0}
Fix a finitely generated commutative monoid~$Q$ and an algebraically
closed field\/~$\kk$ of characteristic~$0$.  Every binomial ideal~$I$
in~$\kk[Q]$ admits a binomial primary decomposition: $I = P_1 \cap
\cdots \cap P_r$ for binomial ideals $P_1,\ldots,P_r$.
\end{cor}
\begin{proof}
Choose a presentation $\NN^n \onto Q$.  The kernel of the induced
presentation $\kk[\xx] \onto \kk[Q]$ is a binomial ideal
in~$\kk[\xx]$.  Therefore the preimage of $I$ in~$\kk[\xx]$ is a
binomial ideal~$I'$.  The image in~$\kk[Q]$ of any binomial primary
decomposition of~$I'$ is a binomial primary decomposition of~$I$.
Therefore it suffices to prove the case where $Q = \NN^n$ and $I' =
I$.  Since every associated prime of~$I$ is binomial by
Theorem~\ref{t:es}, the result follows from
Corollary~\ref{c:primDecomp}.
\end{proof}

Corollary~\ref{c:char0} is stated only for characteristic~$0$ to
demonstrate the connection between prior results in this survey.
However, the restriction is unnecessary.

\begin{thm}\label{t:allChar}
Corollary~\ref{c:char0} holds for fields of positive characteristic,
as well.
\end{thm}
\begin{proof}
For polynomial rings this is \cite[Theorem~7.1]{binomialIdeals}, and
the case of general monoids~$Q$ follows by the argument in the proof
of Corollary~\ref{c:char0}.
\end{proof}

What's missing in the positive characteristic case is combinatorics of
primary ideals.

\begin{prob}
Characterize primary binomial ideals and primary components of
binomial ideals combinatorially in positive characteristic.
\end{prob}

Note, however, that a solution to this problem would still not say how
to discover---from the combinatorics---which primes are associated.
The same is true in characteristic~$0$.  Thus
Corollary~\ref{c:primDecomp} is unsatisfactory for two reasons:
\begin{itemize}
\item%
it requires strong hypotheses on the field~$\kk$; and
\item%
it assumes we know which primes $I_{\rho,J}$ are associated to~$I$.
\end{itemize}
Fortunately, there is a combinatorial, lattice-point method to
recognize associated primes---or at least, to reduce the recognition
to a finite problem.  The main point is Theorem~\ref{t:mesodecomp}:
the combinatorics of the graph~$G_I$ can be used to construct a
decomposition of~$I$ as an intersection of ``primary-like'' binomial
ideals in a manner requiring no hypotheses on the characteristic or
algebraic closure of the base field.
The statement employs some additional concepts.

\begin{defn}
A subset of $\ZZ^n$ is \emph{$J$-bounded} if it intersects only
finitely many cosets of~$\ZZ^J$ in~$\ZZ^n$.
\end{defn}

\begin{lemma}
If $I \subseteq \kk[\xx]$ is a binomial ideal, then $\ZZ^J$ acts on
the set of $J$-bounded
components of the graph~$G_{I[\ZZ^J]}$ on~$\ZZ^J \times \NN^\oJ$
induced by the localization~$I[\ZZ^J]$ along
$\NN^J$.
\end{lemma}
\begin{proof}
In fact, $\ZZ^J$ acts on all of the connected components, because the
Laurent monomials $\xx^u$ for $u \in \ZZ^J$ are units modulo $I[\ZZ^J]
= I[\ZZ F]$ for the face $F = \NN^J$.
\end{proof}

\begin{defn}\label{d:L}
A \emph{witness} for a sublattice $L \subseteq \ZZ^J$
\emph{potentially associated} to~$I$
is any element in a $J$-bounded connected component of~$G_{I[\ZZ^J]}$
whose stabilizer is~$L$.
\end{defn}

\begin{example}
The binomial ideal
$$%
\psfrag{x}{$\scriptstyle x$}
\psfrag{y}{$\scriptstyle y$}
\begin{array}{@{}cc@{}}
  \begin{array}{@{}c@{}}
  I = \<y-x^2y,y^2-xy^2,y^3\> \subseteq \kk[x,y]
  \quad\ \longleftrightarrow\
  \end{array}
&
  \begin{array}{c@{}}\includegraphics{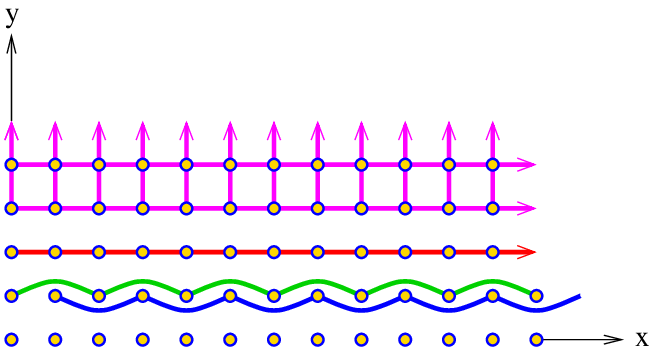}\end{array}
\end{array}
$$
induces the depicted congruence.  Its potentially associated lattices
are all contained in $\ZZ = \ZZ^{\{1\}}$, parallel to the $x$-axis.
The lattices are generated by~$0$, by~$[\twoline 20]$, and
by~$[\twoline 10]$.
\end{example}

The subset~$J$ is part of the definition of potentially associated
sublattice; it is not enough to specify $L$ alone.  The notion of
\emph{associated lattice}, without the adverb ``potentially'', would
require further discussion of primary decomposition of congruences on
monoids; see
the definition of associated lattice in \cite{mesoprimary}.  That
said, the set of potentially associated lattices, which contains the
set of associated ones, suffices for the purposes here, although
sharper results could be stated with the more precise~notion.

\begin{prop}\label{p:witness}
Every binomial ideal $I \subseteq \kk[\xx]$ has finitely many
potentially associated lattices $L \subseteq \ZZ^J$.  If $K =
(I:\xx^u) \subseteq \kk[\xx]$ is the annihilator of\/~$\xx^u$
in~$\kk[\xx]/I$ for a witness $u \in \NN^n$, then $K[\ZZ^J] + \mm_J =
I_{\sigma,J}[\ZZ^J]$ for a uniquely determined \emph{witness
character} $\sigma: L \to \kk^*$.  Given~$I$, each $L \subseteq \ZZ^J$
determines finitely many witness~\mbox{characters}.
\end{prop}
\begin{proof}
This is
proved in \cite{mesoprimary} on the way to the
existence theorem for combinatorial mesoprimary decomposition.  The
finiteness of the set of potentially associated lattices traces back
to the noetherian property for congruences on finitely generated
commutative monoids.  The conclusion concerning~$K$ is little more
than the characterization of binomial ideals in Laurent polynomial
rings~\cite[Theorem~2.1]{binomialIdeals}.  The finiteness of the
number of witness characters occurs because witnesses for $L \subseteq
\ZZ^J$ with distinct witness characters are forced to be incomparable
in~$\NN^n$.
\end{proof}

In Proposition~\ref{p:witness}, the domain $L$ of the
character~$\sigma$ appearing in $I_{\sigma,J}$ need not be saturated
(see Definition~\ref{d:IrhoJ}), and no hypotheses are required on the
field~$\kk$.

Deducing combinatorial statements about associated primes or primary
decompositions of binomial ideals is often most easily accomplished by
reducing to the case of ideals with the simplest possible structure in
this regard.

\begin{defn}\label{d:mesoprimary}
A binomial ideal with a unique potentially associated lattice is
called \emph{mesoprimary}.  A \emph{mesoprimary decomposition} of a
binomial ideal $I \subseteq \kk[\xx]$ is an expression of~$I$ as an
intersection of finitely many mesoprimary binomial ideals.
\end{defn}

\begin{example}\label{e:mesoprimary}
Primary binomial ideals in polynomial rings over algebraically closed
fields of characteristic~$0$ are mesoprimary.  That is basically the
content of Theorem~\ref{t:monPrimary} for such fields, given
Proposition~\ref{p:polyringPrimary'}.  More precisely, the
combinatorics of mesoprimary ideals is just like that of primary
ideals, except that instead of an affine semigroup acting semifreely,
an arbitrary finitely generated cancellative monoid acts semifreely;
see
the characterizations of mesoprimary congruences in
\cite{mesoprimary}.
\end{example}

Definition~\ref{d:mesoprimary} only stipulates constancy of the
combinatorics, not the arithmetic---meaning the witness
characters---but the arithmetic constancy is automatic.

\begin{lemma}
If $I$ is a mesoprimary ideal, then the witnesses for the unique
potentially associated lattice all share the same witness character.
\end{lemma}
\begin{proof}
This follows from the same witness incomparability that appeared in
the proof of Proposition~\ref{p:witness}.  The statement is equivalent
to one direction of the characterization of mesoprimary binomial
ideals as those with precisely one associated mesoprime; see
\cite{mesoprimary}, where a complete proof can be found.
\end{proof}

\begin{thm}\label{t:mesodecomp}
Every binomial ideal $I \subseteq \kk[\xx]$ admits a mesoprimary
decomposition in which the unique associated lattice and witness
character of each mesoprimary component is potentially associated
to~$I$.
\end{thm}
\begin{proof}
This is a weakened form of the existence theorem for combinatorial
mesoprimary decomposition
in \cite{mesoprimary}.
\end{proof}

The power of Theorem~\ref{t:mesodecomp} lies in the crucial conceit
that the combinatorics of the graph~$G_I$ controls everything, so the
lattices associated to the mesoprimary components are severely
restricted.  Over an algebraically closed field of characteristic~$0$,
for instance, every primary decomposition is a mesoprimary
decomposition, but usually the lattices are not associated to~$I$.
This is the case for a lattice ideal~$I_L$, as long as the lattice~$L$
is not saturated: the ideal~$I_L$ is already mesoprimary, but the
associated lattice of every associated prime is the
\emph{saturation}~$L_\sat = (L \otimes_\ZZ \QQ) \cap \ZZ^n$, the
smallest saturated sublattice of~$\ZZ^n$ containing~$L$.  In general,
the combinatorial control is what allows Theorem~\ref{t:mesodecomp} to
be devoid of hypotheses on the field.

As in any expression of an ideal~$I$ as an intersection of larger
ideals, information about associated primes of~$I$ can just as well be
read off of the intersectands.  For mesoprimary decompositions this is
especially effective because primary decomposition of mesoprimary
ideals
\cite{mesoprimary} is essentially as simple as that of lattice ideals
\cite[Corollary~2.5]{binomialIdeals}.  In particular, when the field
is algebraically closed, potentially associated lattices yield
associated primes by way of saturation.  The point is that only
finitely many characters $L_\sat \to \kk^*$ restrict to a given fixed
character $L \to \kk^*$.  In fact, when $\kk$ is algebraically closed,
these characters are in bijection with the finite group
$\Hom(L_\sat/L,\kk^*)$.  Thus Theorem~\ref{t:mesodecomp} reduces the
search for associated primes of~$I$ to the combinatorics of the
graph~$G_I$, along with a minimal amount of arithmetic.

\begin{cor}\label{c:mesoprime}
If the field\/ $\kk$ is algebraically closed, then every associated
prime of\/ $I \subseteq \kk[\xx]$ is $I_{\rho,J}$ for some character
$\rho: L_\sat \to \kk^*$ whose restriction to~$L$ is one of the
finitely many witness characters defined on a potentially associated
lattice $L \subseteq \ZZ^J\!$~of\/~$I$.
\end{cor}
\begin{proof}
Every associated prime of~$I$ is associated to a mesoprime in the
decomposition from Theorem~\ref{t:mesodecomp}.  Now apply either the
primary decomposition of mesoprimary ideals
\cite{mesoprimary} or the witness theorem for cellular ideals
\cite[Theorem~8.1]{binomialIdeals}, using the fact that mesoprimary
ideals are cellular
\cite{mesoprimary}.
\end{proof}

\begin{remark}
In the special case where $I$ is \emph{cellular}, meaning that every
variable is either nilpotent or a nonzerodivisor modulo~$I$,
Corollary~\ref{c:mesoprime} coincides with
\cite[Theorem~8.1]{binomialIdeals}.  Every binomial ideal in any
polynomial ring over any field is an intersection of cellular ideals
\cite[Theorem~6.2]{binomialIdeals}, with at most one cellular
component for each subset $J \subseteq \{1,\ldots,n\}$, so it suffices
for many purposes to understand the combinatorics of cellular ideals.
(Theorem~\ref{t:mesodecomp} strengthens this approach, since
mesoprimary ideals are cellular and their combinatorics is
substantially simpler.)  The way witnesses and witness characters are
defined above, however, it is not quite obvious that the information
extracted from witnesses for the original ideal~$I$ and those for its
cellular components coincides.  That this is indeed the case
constitutes a key ingredient proved in preparation for the existence
theorem for combinatorial mesoprimary decomposition in
\cite{mesoprimary}.
\end{remark}

\begin{exercise}
The ideal $I = \<xz-yz,x^2-x^3\> \subseteq \kk[x,y,z]$ from
Example~\ref{e:localize} has primary decomposition $I =
\<x-y,x^2,xy,y^2\> \cap \<x-1,y-1\> \cap \<x^2, z\> \cap \<x-1, z\>$.
The reader is invited to find all associated lattices $L \subseteq
\ZZ^J$ of~$I$ and match them to the associated primes of~$I$.  Then
the reader can verify, using Corollary~\ref{c:primDecomp}, that the
given primary decomposition of~$I$ really is one.  Hint: take $J \in
\big\{\{3\}, \{1,2,3\}, \{2\}, \{1,2\}\big\}$.
\end{exercise}

The upshot of Sections~\ref{s:affine}--\ref{s:decomp} is that the
lattice-point combinatorics of congruences on monoids lifts to
combinatorics of monomial and binomial primary and mesoprimary
decompositions of binomial ideals in monoid algebras.  {}From there,
binomial primary decomposition is a small arithmetic step, having to
do with group characters for finitely generated abelian groups.

\part{Applications}\vspace{-1ex}

\section{Hypergeometric series}\label{s:horn}

The idea for lattice-point methods in binomial primary decomposition
originated in the study of hypergeometric systems of differential
equations, particularly their series solutions.  The literature on
these systems and series is so vast---owing to its connections with
physics, numerical analysis, combinatorics, probability, number
theory, complex analysis, and algebraic geometry---that one section in
a survey lacks the ability to lend proper perspective.  Therefore, the
goal of this section is to make a beeline for the connections to
binomial primary decomposition, with just enough background along the
way to allow the motivations and conclusions to shine through.  Much
of the exposition is borrowed from \cite{dmm-mega,dmm}, sometimes
nearly verbatim.  The extended abstract \cite{dmm-mega} presents a
broader, more complete historical~overview.

\subsection{Binomial Horn systems}\label{s:horn'}

Horn systems are certain sets of linear partial differential equations
with polynomial coefficients.  Their development grew out of the
ordinary univariate hypergeometric theory going back to Gauss (see
\cite{sk85}, for example) and Kummer \cite{kummer}, through the
bivariate versions of Appell, Horn, and Mellin
\cite{appell,horn89,horn31,mellin}.  These formulations had no
apparent connection to binomials, but through a relatively simple
change of variables, Gelfand, Graev, Kapranov, and Zelevinsky brought
binomials naturally into the picture \cite{ggz,gkz}.

The data required to write down a binomial Horn system consist of a
basis for a sublattice $L \subseteq \ZZ^n$ and a
homomorphism $\beta: \ZZ^n/L \to \CC$.  Focus first on the basis,
which is traditionally arranged in an integer $n \times m$ matrix~$B$,
where $m = \mathrm{rank}(L)$.  If $\bb \in \ZZ^n$ is a column of~$B$,
then $\bb$ determines a binomial:
$$%
  \bb \in \ZZ^n \ \goesto\ \del^{\bb_+} - \del^{\bb_-} \in \CC[\del] =
  \CC[\del_1,\ldots,\del_n],
$$
where $\bb = \bb_+ - \bb_-$ expresses the vector $\bb$ as a difference of
nonnegative vectors with disjoint support.  Elements of the polynomial
ring $\CC[\del]$ are to be viewed as differential operators on
functions $\CC^n \to \CC$.  Therefore the matrix~$B$ determines a
system of $m$ binomial differential operators, one for each column.
The interest is a~priori in solutions to differential systems, not
really the systems themselves, so it is just as well
$$%
  I(B) = \<\del^{\bb_+} - \del^{\bb_-} \mid \bb = \bb_+ - \bb_- \text{ is a
  column of } B\> \subseteq \CC[\del]
$$
generated by these binomials, because any function annihilated by the
$m$ binomials is annihilated by all of~$I(B)$.

\begin{example}\label{e:0123B}
In the 0123 situation from Examples~\ref{e:0123} and~\ref{e:0123'},
using variables $\del = \del_1,\del_2,\del_3,\del_4$ instead of
$a,b,c,d$ or $x_1,x_2,x_3,x_4$ yields $I(B) = \<\del_1\del_3 -
\del_2^2, \del_2\del_4 - \del_3^2\>$.
\end{example}

\begin{example}\label{e:1100B}
In the \1100 situation from Examples~\ref{e:1100} and~\ref{e:1100'}
with $\del$ variables, $I(B) = \<\del_1\del_3 - \del_2, \del_1\del_4 -
\del_2\>$.
\end{example}

The set of homomorphisms $\ZZ^n/L \to \CC$ is a complex vector space
$\Hom(\ZZ^n/L,\CC)$ of dimension $d := n-m$.  Choosing a basis for
this vector space is the same as choosing a basis for $(\ZZ^n/L)
\otimes_\ZZ \CC$, which is the same as choosing a $d \times n$
matrix~$A$ with $AB = 0$.  Let us now, once and for all, fix such a
matrix~$A$ with entries $a_{ij}$ for $i = 1,\ldots,d$ and $j =
1,\ldots,n$.  The situation is therefore just as it was in
Examples~\ref{e:0123'} and~\ref{e:1100'}, and our homomorphism
$\ZZ^n/L \to \CC$ becomes identified with a complex vector $\beta \in
\CC^d$.  Together, $A$ and~$\beta$ determine $d$ differential
operators $E_1 - \beta_1, \ldots, E_d - \beta_d$, where
$$%
  E_i = a_{i1} x_1\del_1 + \cdots + a_{in}x_n\del_n.
$$
Note that $a_{ij}x_j\del_j$ is the operator on functions
$f(x_1,\ldots,x_n): \CC^n \to \CC$ that takes the partial derivative
with respect to~$x_j$ and multiplies the resulting function
by~$a_{ij}x_j$.

\begin{defn}\label{d:horn}
The \emph{binomial Horn system} $H(B,\beta)$ is the system
\begin{align*}
  I(B)f &= 0
\\
  E_1 f &= \beta_1 f
\\
        &\ \,\vdots
\\
  E_d f &= \beta_d f
\end{align*}
of differential equations on functions $f(\xx): \CC^n \to \CC$
determined by the \emph{lattice basis ideal} $I(B)$ and the
\emph{Euler operators} $E_1 - \beta_1, \ldots, E_d - \beta_d$.
\end{defn}

The goal is to find, characterize, or otherwise understand the
solutions to~%
$H(B,\beta)$.

\begin{example}\label{e:0123B'}
In the 0123 case from Example~\ref{e:0123'}, $H(B,\beta)$ has lattice
basis part
\begin{align*}
  I(B)f = 0 \iff\mbox{}& (\del_1\del_3 - \del_2^2)f = 0
\\          \text{and }& (\del_2\del_4 - \del_3^2)f = 0
\end{align*}
and the Euler operators yield the following equations:
\begin{align*}
  (x_1\del_1 + x_2\del_2 + \phantom{2}x_3\del_3 + \phantom{3}x_4\del_4) f &= \beta_1 f
\\(\phantom{x_1\del_1 + \mbox{}}x_2\del_2 + 2x_3\del_3 + 3x_4\del_4) f &= \beta_2 f.
\end{align*}
\end{example}

\begin{example}\label{e:1100B'}
In the \1100 case from Example~\ref{e:1100B}, $H(B,\beta)$ has lattice
basis part
\begin{align*}
  I(B)f = 0 \iff\mbox{}& (\del_1\del_3 - \del_2)f = 0
\\          \text{and }& (\del_1\del_4 - \del_2)f = 0
\end{align*}
and the Euler operators yield the following equations:
\begin{align*}
  (x_1\del_1 + x_2\del_2 \phantom{\mbox{}+ x_3\del_3 + x_4\del_4}) f &= \beta_1 f
\\(\phantom{x_1\del_1 + \mbox{}}x_2\del_2 + x_3\del_3 + x_4\del_4) f &= \beta_2 f.
\end{align*}
\end{example}

Since Horn systems are linear, their solution spaces are complex
vector spaces.  More precisely, the term \emph{solution space} in what
follows means the vector space of local holomorphic solutions defined
in a neighborhood of a (fixed, but arbitrary) point in~$\CC^n$ that is
nonsingular for the Horn system.

\begin{example}\label{e:puiseux}
In the 0123 case from Example~\ref{e:0123B'}, for any parameter
vector~$\beta$, the Puiseux monomial $f =
x_1^{\beta_1/3}x_4^{\beta_2/3}$ is a solution of~$H(B,\beta)$.
Indeed,
$$%
  \del_1\del_3(f) =
  \del_2(f) =
  \del_2\del_4(f) =
  \del_3^2(f) = 0,
$$
so $I(B)f = 0$, and $(E_1 - \beta_1)f = (E_2 - \beta_2) f = 0$ because
\begin{align*}
  x_1\del_1(f) &= (\beta_1 - {\textstyle\frac13}\beta_2)f
\\x_2\del_2(f) &= 0
\\x_3\del_3(f) &= 0
\\x_4\del_4(f) &= {\textstyle\frac13}\beta_2 f.
\end{align*}
Erd\'elyi produced this solution and similar ones in other examples
\cite{erdelyi}, but he furnished no explanation for why it should
exist or how he found~it.  In this particular example, the Horn system
has, in addition to the Puiseux monomial~$f$, three linearly
independent fully supported solutions, in the following sense.
\end{example}

\begin{defn}
A Puiseux series solution $f$ to a Horn system $H(B,\beta)$ is
\emph{fully supported} if there is a
normal affine semigroup~$Q$ of dimension~$m$ and a vector $\gamma \in
\CC^n$ such that the translate
$\gamma + Q$ consists of vectors that are exponents on monomials with
nonzero coefficient in~$f$.
\end{defn}

The integer $m = \mathrm{rank}(L)$ in the definition is the maximum
possible: the Euler operator equations impose homogeneity on Puiseux
series solutions, meaning that every solution must be supported on a
translate of $L \otimes_\ZZ \CC$.  In fact, the translate is by any
vector $\gamma \in \CC^n$ satisfying $A\gamma = \beta$.

\begin{questions}\label{q}
Consider the family of Horn systems determined by~$B$ with
varying~$\beta$.
\begin{enumerate}
\item\label{q:finite-rank}%
For which parameters~$\beta$ does $H(B,\beta)$ have finite-dimensional
solution space?

\item\label{q:rank}%
What is a combinatorial formula for the minimum solution space
dimension, over all possible choices of the parameter~$\beta$?

\item\label{q:generic}%
Which
$\beta$ are generic in the sense that the minimum dimension is
attained?

\item\label{q:support}%
Which monomials occur in solutions expanded as
series centered at the origin?
\end{enumerate}
\end{questions}

These
questions arise from classical work done in the 1950s, such as
Erd\'elyi's, and earlier.  Implicit in Question~\ref{q:generic} is
that the dimension of the solution space rises above the minimum for
only a ``small'' subset of parameters~$\beta$.

\begin{example}\label{e:uncountable}
In the \1100 case from Example~\ref{e:1100B'}, if $\beta_1 = 0$, then
any (local holomorphic) bivariate function $f(x_3,x_4)$ satisfying
$x_3\del_3 f + x_4\del_4 f = \beta_2 f$ is a solution of the Horn
system~$H(B,\beta)$.  The space of such functions is
infinite-dimensional; in fact, it has uncountable dimension, as it
contains all Puiseux monomials $x_3^{w_3} x_4^{w_4}$ with $w_3,w_4 \in
\CC$ and $w_3+w_4 = \beta_2$.  When $\beta_1 \neq 0$, the solution
space has finite dimension.
\end{example}

The \1100 example has vast numbers of linearly independent solutions
expressible as Puiseux series with small support, but only for special
values of~$\beta$.  In contrast,
in the 0123 case there are many fewer series solutions of small
support, but they appear for arbitrary values of~$\beta$.  This
dichotomy is central to the interactions of Horn systems with binomial
primary decomposition.

\subsection{True degrees and quasidegrees of graded modules}\label{s:degrees}

The commutative algebraic
version of the
dichotomy just mentioned arises from elementary (un)boundedness
of Hilbert functions of $A$-graded modules (recall
Definition~\ref{d:A-graded}, Example~\ref{e:A-graded}, and
Lemma~\ref{l:hilb}); see Definition~\ref{d:toralAndean}.  For the
remainder of this section, fix a matrix $A \in \ZZ^{d \times n}$ of
rank~$d$ whose affine semigroup $\NN A \subseteq \ZZ^d$ is pointed
(Definition~\ref{d:ruleset}).

\begin{lemma}
A binomial ideal $I \subseteq \CC[\del]$ is $A$-graded if and only if
it is generated by binomials $\xx^\uu - \lambda\xx^\vv$ for which
$A\uu = A\vv$.\qed
\end{lemma}

It is of course not necessary---and it almost never happens---that
every binomial of the form $\xx^\uu - \lambda\xx^\vv$ with $A\uu =
A\vv$ lies in~$I$.

\begin{example}
$I = I(B)$ is always $A$-graded when $B$ is a matrix for~$\ker(A)$.
\end{example}

The set of degrees where a graded module is nonzero should, for the
purposes of the applications to Horn systems, be considered
geometrically.

\begin{defn}
For any $A$-graded module~$M$,
$$%
  \tdeg(M) = \{\alpha \in \ZZ^d \mid M_\alpha \neq 0\}
$$
is the set of \emph{true degrees} of~$M$.  The set of
\emph{quasidegrees} of~$M$ is
$$%
  \qdeg(M) = \ol{\tdeg(M)},
$$
the Zariski closure in~$\CC^d$ of the true degree set of~$M$.
\end{defn}

The Zariski closure here warrants some discussion.  By definition, the
\emph{Zariski closure} of a subset $T \subseteq \CC^d$ is the largest
set $\oT$ of points in~$\CC^d$ such that every polynomial vanishing
on~$T$ also vanishes on~$\oT$.  All of the sets $T$ that we shall be
interested in are sets of lattice points in $\ZZ^d \subseteq \CC^d$.
When $T$ consists of lattice points on a line, for example, its
Zariski closure~$\oT$ is the whole line precisely when $T$ is
infinite.  When $T$ is contained in a plane, its Zariski closure is
the whole plane only if $T$ is not contained in any algebraic curve in
the plane.  In the cases that interest us, $\oT$~will always be a
finite union of translates of linear subspaes of~$\CC^d$.

\begin{lemma}\label{l:Z}
If $M$ is a finitely generated $A$-graded module over~$\CC[\del]$,
then $\qdeg(M)$ is a finite arrangement of affine subspaces
of~$\CC^d$, each one parallel to~$\ZZ A_J$ for some $J \subseteq
\{1,\ldots,n\}$, where $A_J$ is the submatrix of~$A$ comprising the
columns indexed by~$J$.
\end{lemma}
\begin{proof}
This is \cite[Lemma~2.5]{dmm}.  Since $M$ is noetherian, it has a
finite filtration whose successive quotients are $A$-graded translates
of~$\CC[\del]/\pp$ for various $A$-graded primes~$\pp$.  Finiteness of
the filtration implies that $\qdeg(M)$ is the union of the quasidegree
sets of these $A$-graded translates.  But $\tdeg(\CC[\del]/\pp)$ is
the affine semigroup generated by $\{\deg(\del_i) \mid \del_i \notin
\pp\}$, because $\tdeg(\CC[\del]/\pp)$ is an integral domain.
\end{proof}

\begin{example}
If $\pp = I_L + \mm_J$ is a binomial prime ideal, then
$\qdeg(\CC[\del]/\pp) = \CC A_J$ is the complex vector subspace
of~$\CC^d$ spanned by the columns of~$A$ indexed by~$J$.
\end{example}

\begin{example}\label{e:tdeg}
In the \1100 case from Examples~\ref{e:1100} and~\ref{e:1100'},
consider the $A$-graded module $M =
\CC[\del_1,\del_2,\del_3,\del_4]/I$ for varying ideals~$I$.  Then the
true degree sets and quasidegree sets can be depicted as follows.
\begin{enumerate}
\item%
When $I = I_A = \<\del_1\del_3 - \del_2, \del_3 - \del_4\>$, so that
$M \cong \CC[\del_1,\del_3]$ by the homomorphism sending $\del_2
\mapsto \del_1\del_3$ and $\del_4 \mapsto \del_3$,
$$%
\psfrag{1}{}
\psfrag{d1}{\tiny\color{darkpurple}$\del_1$}
\psfrag{d3}{\tiny\color{darkpurple}$\del_3$}
\begin{array}{cc}
  \begin{array}{l}
  \\
  M = \CC[\del]/\<\del_1\del_3 - \del_2, \del_3 - \del_4\> \cong
  \CC[\del_1,\del_3]
  \\[1ex]
  {\color{darkpurple}\deg(\del_1) = \bigl[\twoline 10\bigr]} \text{ and }
  {\color{darkpurple}\deg(\del_3) = \bigl[\twoline 01\bigr]}
  \\[1ex]
  {\color{blue}\tdeg(M) = \NN\bigl\{\bigl[\twoline 10\bigr],
  \bigl[\twoline 01\bigr]\bigr\} = \NN A_{\{1,3\}}}
  \\[1ex]
  {\color{darkgreen} \qdeg(M) = \CC^2}
  \end{array}
  \ \ \longleftrightarrow\ \
&
  \begin{array}{c}\includegraphics{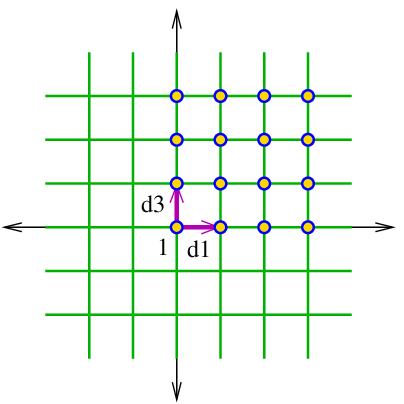}\end{array}
\\\mbox{}
\end{array}
$$

\item%
When $I = \<\del_1,\del_2\>$, so that $M \cong \CC[\del_3,\del_4]$ by
the homomorphism sending $\del_1 \mapsto 0$ and $\del_2 \mapsto 0$,
$$%
\psfrag{1}{}
\psfrag{d1}{}
\psfrag{d3}{\tiny\color{darkpurple}$\del_3,\del_4$}
\begin{array}{cc}
  \begin{array}{l}
  \\
  M = \CC[\del]/\<\del_1,\del_2\> \cong \CC[\del_3,\del_4]
  \\[1ex]
  {\color{darkpurple}\deg(\del_3) = \bigl[\twoline 01\bigr]} \text{ and }
  {\color{darkpurple}\deg(\del_4) = \bigl[\twoline 01\bigr]}
  \\[1ex]
  {\color{blue}\tdeg(M) = \NN\bigl\{\bigl[\twoline 01\bigr]\bigr\} =
  \NN A_{\{3,4\}}}
  \\[1ex]
  {\color{darkgreen} \qdeg(M) = \text{vertical axis in }\CC^2}
  \end{array}
  \ \ \ \longleftrightarrow\ \
&
  \begin{array}{c}\includegraphics{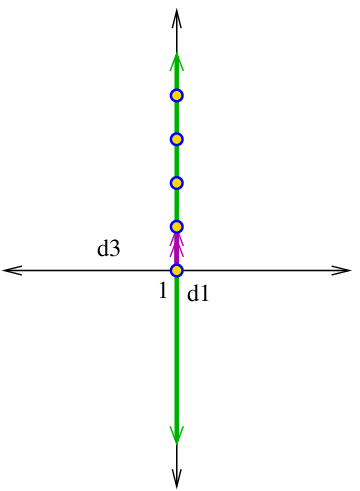}\end{array}
\\\mbox{}
\\[-2ex]\mbox{}
\end{array}
$$

\item%
When $I = \<\del_1^2,\del_2\>$, so that $M \cong \CC[\del_3,\del_4]
\oplus \del_1\CC[\del_3,\del_4]$,
$$%
\psfrag{1}{}
\psfrag{d1}{\tiny\color{darkpurple}$\del_1$}
\psfrag{d3}{\tiny\color{darkpurple}$\del_3,\del_4$}
\begin{array}{@{}cc@{}}
\\[-4ex]
  \begin{array}{@{}l}
  \\
  M = \CC[\del]/\<\del_1^2,\del_2\> \cong
  \CC[\del_3,\del_4] \oplus \del_1\CC[\del_3,\del_4]
  \\[1ex]
  {\color{darkpurple}\deg(\del_1) = \bigl[\twoline 10\bigr]} \text{ and }
  {\color{darkpurple}\deg(\del_3) = \bigl[\twoline 01\bigr]} \text{ and }
  {\color{darkpurple}\deg(\del_4) = \bigl[\twoline 01\bigr]}
  \\[1ex]
  {\color{blue}\tdeg(M) = \NN A_{\{3,4\}} \cup \bigl([\twoline
  10\bigr] + \NN A_{\{3,4\}}\bigr)}
  \\[1ex]
  {\color{darkgreen} \qdeg(M) = \text{vertical axis} \cup
  \bigl(\bigl[\twoline 10\bigr]+\text{vertical axis}\bigr)\text{ in }\CC^2}
  \end{array}
  \ \longleftrightarrow\ \
&
  \begin{array}{c@{}}\includegraphics{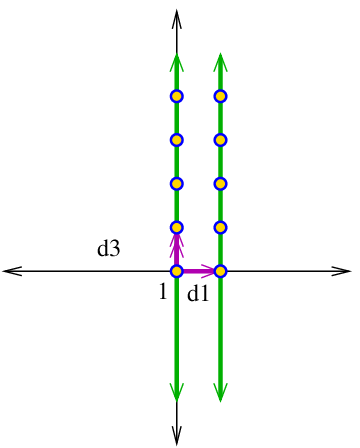}\end{array}
\end{array}
$$
\end{enumerate}
\end{example}

\begin{defn}\label{d:toralAndean}
An $A$-graded prime ideal $\pp \subseteq \CC[\del]$ is
\begin{itemize}
\item%
\emph{toral} if the Hilbert function $\uu \mapsto
\dim_\CC\,(\CC[\del]/\pp)_\uu$ is bounded for $\uu \in \ZZ^d$, and
\item%
\emph{Andean} if the Hilbert function is unbounded.
\end{itemize}
\end{defn}

The adjective ``Andean'' indicates that Andean $A$-graded components
sit like a high, thin mountain range on~$\ZZ^d$, of unbounded
elevation, over cosets of sublattices~$\ZZ A_J$.

\begin{example}\label{e:1100toralAndean}
Consider the situation from Example~\ref{e:tdeg}.
\begin{enumerate}
\item%
The prime ideal $I_A$ in Example~\ref{e:tdeg}.1 is toral because the
Hilbert function of~$\CC[\del]/I_A$ only takes the value~$1$ on the
true degrees.  By definition, the Hilbert function vanishes outside of
the true degree set.
\item%
The prime ideal $I = \<\del_1,\del_2\>$ in Example~\ref{e:tdeg}.2 is
Andean because the Hilbert function of~$\CC[\del]/I$ is unbounded:
$\bigl[\twoline 0k\bigr] \mapsto k$.
\end{enumerate}
\end{example}

\begin{example}\label{e:0123toral}
In the 0123 situation from Examples~\ref{e:0123} and~\ref{e:0123'},
the ideal $I_A = \<\del_1\del_3 - \del_2^2, \del_2\del_4 - \del_3^2,
\del_1\del_4 - \del_2\del_3\>$ is toral under the $A$-grading.
Indeed, for any matrix~$A$ the toric ideal $I_A$ is toral under the
$A$-grading by Lemma~\ref{l:hilb}: the $Q$-graded Hilbert function of
a monoid algebra~$\kk[Q]$ takes the constant value~$1$ and is hence
bounded.
\end{example}

\begin{thm}\label{t:toralAndean}
Fix an $A$-graded ideal $I \subseteq \CC[\del]$.  Given that $\NN A$
is pointed, every associated prime of~$I$ is $A$-graded, and $I$
admits a decomposition as an intersection of $A$-graded primary
ideals.  The intersection $I_\andean$ of the primary components of~$I$
with Andean associated primes is well-defined.
\end{thm}
\begin{proof}
The $A$-graded conclusion on the associated primes is
\cite[Proposition~8.11]{cca}.  The Andean part is well-defined because
if $\pp \supseteq \qq$ for some Andean~$\pp$ then $\qq$ is also
Andean.  (``The set of Andean primes is closed under going down.'')
\end{proof}

\begin{cor}\label{c:toralAndean}
If $I \subseteq \CC[\del]$ is an $A$-graded ideal, then $I$ admits a
decomposition
$$%
I = I_\toral \cap I_\andean
$$
into toral and Andean parts, where $I_\toral$ is the intersection of
the primary components of~$I$ with toral associated primes in any
fixed primary decomposition~of~$I$.
\end{cor}

\begin{defn}
The \emph{Andean arrangement} of an ideal $I$ is
$\qdeg(\CC[\del]/I_\andean)$.
\end{defn}

The Andean arrangement is a union of affine subspaces of~$\CC^d$ by
Lemma~\ref{l:Z}.  It is well-defined by Theorem~\ref{t:toralAndean}.

\begin{example}\label{e:andean}
The \1100 lattice ideal $I(B) = \<\del_1\del_3 - \del_2, \del_1\del_4
- \del_2\>$ from Example~\ref{e:1100B} has primary decomposition
$$%
\begin{array}{r@{\ }c@{\ }c@{\ }c@{\ }c@{\ }}
  I(B) &=& I_A &\cap& \<\del_1,\del_2\>
\\     &=& I_\toral &\cap& I_\andean,
\end{array}
$$
with toral part $I_\toral = I_A$ and Andean part $I_\andean =
\<\del_1,\del_2\>$ by Example~\ref{e:1100toralAndean}.  Therefore the
Andean arrangement of~$I(B)$ is the thick vertical line in
Example~\ref{e:tdeg}.2.
\end{example}

\subsection{Counting series solutions}\label{s:sols}

The distinction between toral and Andean primes provides the framework
for the answers to Questions~\ref{q}.  Throughout the remainder of
this section, fix a matrix $B \in \ZZ^{n \times m}$ of rank $m = n -
d$ such that $AB = 0$.  Assume that $B$ is \emph{mixed}, meaning that
every nonzero integer vector in the span of the columns of~$B$ has two
nonzero entries of opposite~sign.  The mixed condition is a technical
hypothesis arising while constructing series solu\-tions to
$H(B,\beta)$; its main algebraic consequence is that it forces $\NN A$
to be pointed.

The \emph{multiplicity} of a prime ideal~$\pp$ in an ideal~$I$ is, by
definition, the length of the largest submodule of finite length in
the localization $\CC[\del]_\pp/I_\pp$.  This number is nonzero
precisely when $\pp$ is associated to~$I$.  Combinatorially, when
$\pp$ and $I$ are binomial ideals, the multiplicity of~$\pp$ in~$I$
counts connected components of graphs related to~$G_I$, such as those
in Corollary~\ref{c:primDecomp}.  For the purposes of Horn systems,
the most relevant number is derived from multiplicities of prime
ideals as follows.

\begin{defn}
The \emph{multiplicity} $\mu(L,J)$ of a saturated sublattice $L
\subseteq \ZZ^J$ is the product $\iota\mu$, where $\iota$ is the index
$|L/(\ZZ B\cap\ZZ^J)|$ of the sublattice $\ZZ B \cap \ZZ^J$ in~$L$,
and $\mu$ is the multiplicity of the binomial prime ideal $I_L +
\mm_J$ in the lattice ideal~$I(B)$.
\end{defn}

The factor $\iota = |L/(\ZZ B\cap\ZZ^J)|$ counts the number of partial
characters $\rho : L \to \CC^*$ for which $I_{\rho,J}$ is associated
to~$I(B)$.  It is the penultimate combinatorial input required to
count solutions to Horn systems.  The final one is polyhedral.

\begin{defn}
For any subset $J \subseteq \{1,\ldots,n\}$, write $\vol(A_J)$ for the
\emph{volume} of the convex hull of~$A_J$ and the origin, normalized
so a lattice simplex in~$\ZZ A_J$ has~volume~$1$.
\end{defn}

\begin{example}\label{e:vol}
In the 0123 case, $\vol(A) = 3$, since the columns of $A$
span~$\ZZ^2$, and the convex hull of~$A$ with the origin is a triangle
that is a union of three lattice triangles.
$$%
\psfrag{vol(A) = 3}{\small$\vol(A) = 3$}
\psfrag{vol(A_J) = 1}{\small$\vol(A_{\{1,4\}}) = 1$}
\begin{array}{c@{}}\\[-2.8ex]\includegraphics{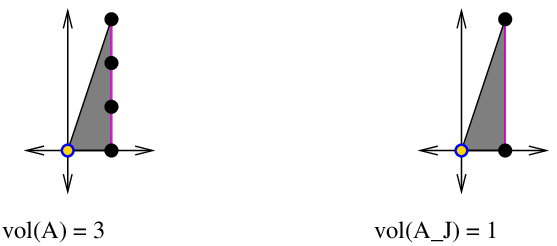}\end{array}
$$
When $J = \{1,4\}$, in contrast, $\vol(A_J) = 1$, since the first and
last columns of~$A$ form a basis for the lattice (of index~$3$
in~$\ZZ^2$) that they span.
\end{example}

\begin{answers}\label{a}
The answers to Questions~\ref{q} for the systems $H(B,\beta)$ are as
follows.
\begin{enumerate}
\item
The dimension is finite exactly when $\beta$ lies in the Andean
arrangement of~$I(B)$.

\item
The generic (minimum) dimension is $\sum \mu(L,J) \cdot \vol(A_J)$,
the sum being over all saturated $L \subseteq \ZZ^J$ such that $I_L +
\mm_J$ is a toral binomial prime with $\CC A_J = \CC^d$.

\item
The minimum rank is attained precisely when $\beta$ lies outside of an
affine subspace arrangement determined by certain local cohomology
modules, with the same flavor as (and containing) the Andean
arrangement.

\item\label{a:support}
If the configuration $A$ lies in an affine hyperplane not containing
the origin, and $\beta$ is general, then the solution space of
$H(B,\beta)$ has a basis containing precisely $\sum_J \mu(L,J) \cdot
\vol(A_J)$ Puiseux series supported on finitely many cosets of~$L$.
\end{enumerate}
\end{answers}

\begin{example}
To illustrate Answer~\ref{a}.1 in the \1100 case, compare
Example~\ref{e:uncountable} to Example~\ref{e:andean}: the solution
space has finite dimension precisely when the parameter lies off the
vertical axis, which is the Andean arrangement in this case.
\end{example}

\begin{example}\label{e:sol0123}
In contrast, both associated primes are toral in the 0123 case, where
$$%
  I(B) = I_A \cap \<\del_2,\del_3\>.
$$
Indeed, the quotient of~$\CC[\del]$ modulo each of these components is
an $A$-graded affine semigroup ring, with the second component
yielding $\CC[\del]/\<\del_2,\del_3\> \cong \CC[\NN A_{\{1,4\}}]$.  It
follows that the solution space has finite dimension for all
parameters~$\beta$.

On the other hand, Answer~\ref{a}.2 is interesting in this 0123 case:
Example~\ref{e:vol} implies that $H(B,\beta)$ has generic solution
space of dimension
$$%
  3 + 1 = \mu\bigl(\ZZ B, \{1,2,3,4\}\bigr) + \mu\bigl(\{0\},
  \{1,4\}\bigr),
$$
with the first summand giving rise to solution series of full support,
and the second summand giving rise to one solution series with finite
support---that is, supported on finitely many cosets of~$\{0\}$---by
Answer~\ref{a}.4.  Compare Example~\ref{e:puiseux}.
\end{example}

\begin{proof}[Proof of Answers~\ref{a}]
These are some of the main results of \cite{dmm}, namely:
\begin{enumerate}
\item%
Theorem~6.3.
\item%
Theorem~6.10.
\item%
Definition~6.9 and Theorem~6.10.
\item%
Theorem~6.10, Theorem~7.14, and Corollary~7.25.
\end{enumerate}
The basic idea is to filter $\CC[\del]/I(B)$ with successive quotients
that are $A$-graded translates of $\CC[\del]/\pp$ for various binomial
primes~$\pp$.  There is a functorial (``Euler--Koszul'') way to lift
this to a filtration of a corresponding $D$-module canonically
constructed from $H(B,\beta)$.  The successive quotients in this
lifted filtration are \emph{$A$-hypergeometric systems} of Gelfand,
Graev, Kapranov, and Zelevinsky \cite{ggz,gkz}.  The solution space
dimension equals the volume for such hypergeometric systems, and the
factor $\mu(L,J)$ simply counts how many times a given such
hypergeometric system appears as a successive quotient in the
$D$-module filtration.  That, together with series solutions
constructed by GGKZ, proves Answers~2 and~4.

When a successive quotient is $\CC[\del]/\pp$ for an Andean
prime~$\pp$, the Euler operators span a vector space of too small
dimension; they consequently fail to cut down the solution space to
finite dimension: at least one ``extra'' Euler operator is needed.
Without this extra Euler operator, its (missing) continuous parameter
allows an uncountable family of solutions as in
Example~\ref{e:uncountable}; this proves Answer~1.  Answer~3 is really
a corollary of the main results of~\cite{mmw}, and is beyond the scope
of~this~survey.
\end{proof}

To explain Erd\'elyi's observation (Example~\ref{e:puiseux}), note
that only $|\ker A/\ZZ B| \cdot \vol(A)$ many of the series solutions
from Answer~\ref{a}.4 have full support, where $\ker(A) = (\ZZ
B)_\sat$ is the saturation of the image of~$B$.  The remaining
solutions have smaller support.  Most lattice basis ideals have
associated primes other than~$I_{\ZZ B}$ \cite{hosten-shapiro}, so
most Horn systems $H(B,\beta)$ have spurious solutions, whether they
be of the toral kind (finite-dimensional, but small support, perhaps
for special paramaters~$\beta$) or Andean kind (uncountable
dimensional).

The ``hyperplane not containing the origin'' condition in
Answer~\ref{a}.4 amounts to a homogeneity condition on~$I(B)$: the
generators should be homogeneous under the standard $\NN$-grading
of~$\CC[\del]$, as in Example~\ref{e:sol0123}.  More deeply, this
condition is equivalent to regular holonomicity of the corresponding
$D$-module \cite{uli06}.  As soon as the support of a Puiseux series
solution is specified, hypergeometric recursions determine the
coefficients up to a global scalar.

The recursive rules governing the coefficients of series solutions to
Horn hypergeometric systems force the combinatorics of lattice-point
graphs upon binomial primary decomposition of lattice basis ideals,
via the arguments in the proof of Answers~\ref{a}.  Granted the
generality of Sections~\ref{s:affine}--\ref{s:decomp}, it subsequently
follows that the questions as well as the answers work essentially as
well for arbitrary $A$-graded binomial ideals, with little adjustment
\cite{dmm}.

\section{Combinatorial games}\label{s:games}

Combinatorial games are two-player affairs in which the sides
alternate moves, both with complete information and no element of
chance.  The germinal goal of Combinatorial Game Theory (CGT) is to
find strategies for such games.  After briefly reviewing the
foundations and history of CGT using some key examples
(Section~\ref{s:cgt}), this section provides an overview of how to
phrase the theory in terms of lattice points in polyhedra
(Section~\ref{s:lattice}).  Exploring data structures for strategies
as generating functions (Section~\ref{s:ratstrat}) or in terms of
mis\`ere quotients (Section~\ref{s:misere}) leads to conjectures and
computational open problems involving binomial ideals and related
combinatorics.

\subsection{Introduction to combinatorial game theory}\label{s:cgt}

There are many different ways to represent games and winning
strategies by combinatorial structures.  To understand their formal
definitions, it is best to have in mind some concrete examples.

\begin{example}\label{e:nim}
The quintessential combinatorial game is \nim.  The players---you
and~I, say---are presented with a finite number of heaps of beans,
such as
$$%
\begin{array}{@{}c@{}}\includegraphics{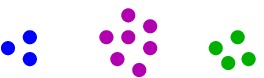}\end{array}
$$
when there are three heaps, of sizes {\color{blue}$3$},
{\color{darkpurple}$7$}, and~{\color{darkgreen}$4$}.  Any finite set
of heaps is a \emph{position} in the game of \nim.  The game is played
by alternating turns, where each turn consists of picking one of the
heaps and removing at least one bean from it.  For instance, if you
play first, then you could remove one bean from the
{\color{blue}$3$}-heap, or three beans from the
{\color{darkgreen}$4$}-heap, or all of the beans from the
{\color{darkpurple}$7$}-heap, to get one of the following positions:
$$%
\begin{array}{@{}c@{}}\includegraphics{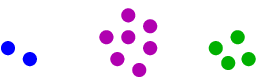}\end{array}
\qquad\text{\quad or\quad}\qquad
\begin{array}{@{}c@{}}\includegraphics{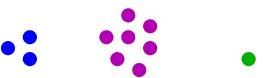}\end{array}
\qquad\text{\quad or\quad}\qquad
\begin{array}{@{}c@{}}\includegraphics{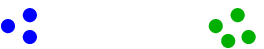}\end{array}
$$

The goal of the game is to play last.  As it turns out, if I play
first in the
{\color{blue}$3$}--{\color{darkpurple}$7$}--{\color{darkgreen}$4$}
game, then you can always force a win by ensuring that you play last.
How?  Take the \emph{nim sum} of the heap sizes: express each heap
size in binary and add these binary numbers digit by digit, as
elements of the field~$\FF_2$ of cardinality~$2$:
$$%
\begin{array}{ccc}
       & 1 & 1
  \\ 1 & 1 & 1
  \\ 1 & 0 & 0
  \\\hline
     0 & 0 & 0
\end{array}
$$
Whatever move I make will alter only one of the summands and hence
will leave a nonzero nim sum, at which point you can always remove
beans from a heap to reset the nim sum to zero; I~can't win because
removing the last bean leaves a zero nim sum.  This general solution
to \nim\ is one of the oldest formal contributions to combinatorial
game theory \cite{nim}.  More general ``heap games'', in which players
take beans from heaps according to specified rules, constitute a core
class of examples~for~the~theory.
\end{example}

\begin{example}\label{e:dawson}
\chess\ and its variants give rise to a rich bounty of combinatorial
games.  One of the most famous, other than \chess\ itself, is \dawson,
played on a $3 \times d$ board with an initial position of opposing
pawns facing one another with a blank rank (row) in between
\cite{dawson}.  When $d = 7$, here is the initial position:
$$%
\begin{array}{@{}c@{}}\includegraphics[height=17mm]{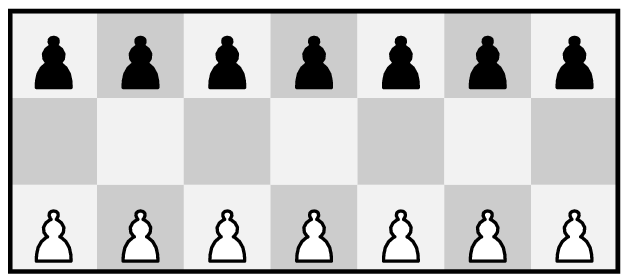}\end{array}
$$
Moves in \dawson\ are the usual moves of \chess\ pawns.  The only
additional rule is that a capture must be made if one is possible.
For example, if white moves first (as usual) and chooses to push the
third pawn, then the game could begin as~follows:
$$%
\begin{array}{@{}c@{}}
\\[-2ex]
\begin{array}{@{}c@{}}\includegraphics[height=17mm]{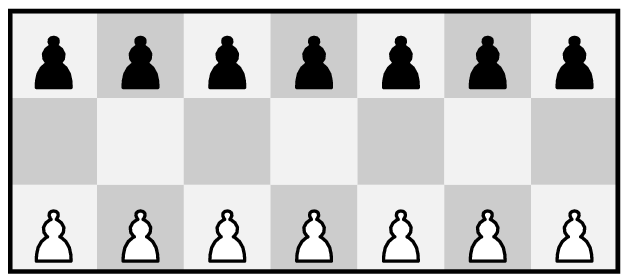}\end{array}
\goesto
\begin{array}{@{}c@{}}\includegraphics[height=17mm]{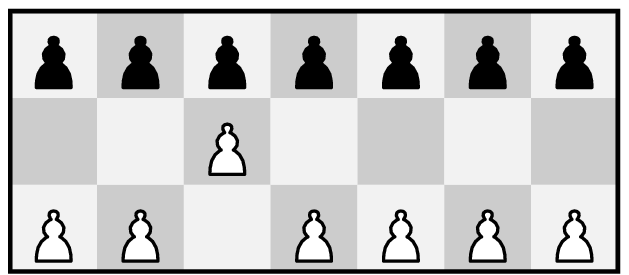}\end{array}
\goesto
\begin{array}{@{}c@{}}\includegraphics[height=17mm]{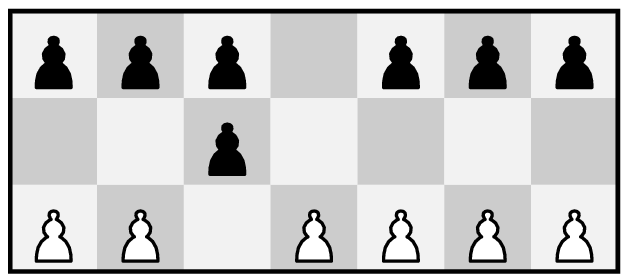}\end{array}
\\[6ex]\qquad
\goesto
\begin{array}{@{}c@{}}\includegraphics[height=17mm]{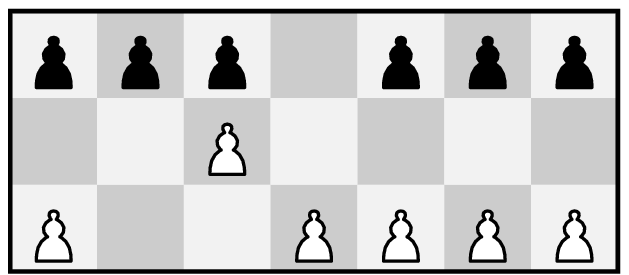}\end{array}
\goesto
\begin{array}{@{}c@{}}\includegraphics[height=17mm]{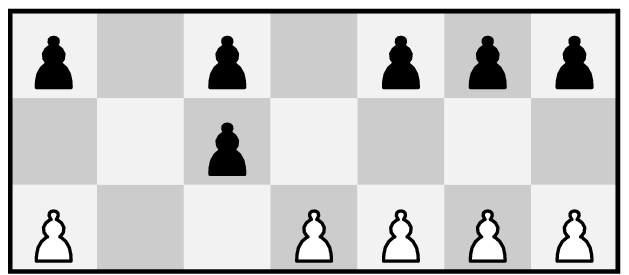}\end{array}
\goesto
\begin{array}{@{}c@{}}\includegraphics[height=17mm]{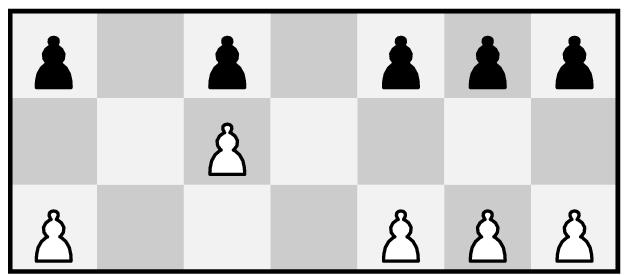}\end{array}
\end{array}
$$%
And now it is black's turn; with no captures available, black is
allowed to push any pawn in file (column) $1$, $5$, $6$, or~$7$ down
to the middle rank.  Another bloodbath ensues, and then it is white's
turn to push a pawn freely.

The goal of \dawson\ is to force your opponent to play last.  Thus, in
contrast to \nim, a player has won when their turn arrives and no
moves are~available.
\end{example}

\begin{defn}\label{d:gameGraph}
An \emph{impartial combinatorial game} is a rooted directed graph on
which two players alternate \emph{moves} along edges.  A player
\emph{wins} by moving to a node (\emph{position}) with no outgoing
edges.  A game is \emph{finite} if its graph is finite and has no
directed cycles.
\end{defn}

The root of the directed graph corresponds to the \emph{initial
position}.  Any finite impartial game can equivalently be represented
as a rooted tree in which the children of each position correspond to
its \emph{options}: the endpoints of its outgoing edges in any
directed graph representation.  A given position might be repeated in
the tree representation.  Given a tree representation, an optimally
efficient directed graph representation can be constructed by
identifying all vertices whose descendant subtrees are isomorphic.

\begin{example}\label{e:nimTree}
In the language of combinatorial game theory, \nim\ is not a finite
game but a family of finite games that specifies a consistent ``rule
set'' for how to play starting from any particular initial position
among an infinite number of possibilities.  In the situation of
Example~\ref{e:nim}, using $ijk$ for $0 \leq i,j,k \leq 9$ to
represent three heaps of sizes $i$, $j$, and~$k$, the top of the game
tree is
$$%
\psfrag{374}{$\scriptstyle {\color{blue}3}{\color{darkpurple}7}{\color{darkgreen}4}$}
\psfrag{074}{$\scriptstyle {\color{blue}0}{\color{darkpurple}7}{\color{darkgreen}4}$}
\psfrag{174}{$\scriptstyle {\color{blue}1}{\color{darkpurple}7}{\color{darkgreen}4}$}
\psfrag{274}{$\scriptstyle {\color{blue}2}{\color{darkpurple}7}{\color{darkgreen}4}$}
\psfrag{304}{$\scriptstyle {\color{blue}3}{\color{darkpurple}0}{\color{darkgreen}4}$}
\psfrag{314}{$\scriptstyle {\color{blue}3}{\color{darkpurple}1}{\color{darkgreen}4}$}
\psfrag{372}{$\scriptstyle {\color{blue}3}{\color{darkpurple}7}{\color{darkgreen}2}$}
\psfrag{373}{$\scriptstyle {\color{blue}3}{\color{darkpurple}7}{\color{darkgreen}3}$}
\psfrag{004}{$\scriptstyle {\color{blue}0}{\color{darkpurple}0}{\color{darkgreen}4}$}
\psfrag{014}{$\scriptstyle {\color{blue}0}{\color{darkpurple}1}{\color{darkgreen}4}$}
\psfrag{073}{$\scriptstyle {\color{blue}0}{\color{darkpurple}7}{\color{darkgreen}3}$}
\psfrag{302}{$\scriptstyle {\color{blue}3}{\color{darkpurple}0}{\color{darkgreen}2}$}
\psfrag{303}{$\scriptstyle {\color{blue}3}{\color{darkpurple}0}{\color{darkgreen}3}$}
\psfrag{371}{$\scriptstyle {\color{blue}3}{\color{darkpurple}7}{\color{darkgreen}1}$}
\psfrag{...}{\small\,$\cdots$}
\psfrag{.}{\small$\vdots$}
\begin{array}{@{}c@{}}\includegraphics{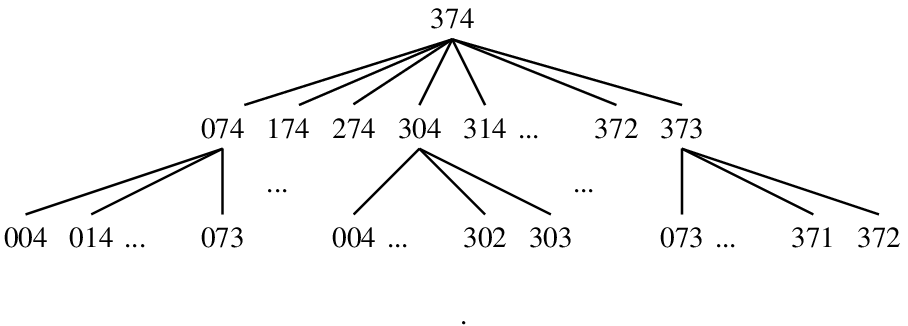}\end{array}
$$
Note that positions such as
${\color{blue}3}{\color{darkpurple}7}{\color{darkgreen}2}$,
${\color{blue}0}{\color{darkpurple}0}{\color{darkgreen}4}$, and
${\color{blue}0}{\color{darkpurple}7}{\color{darkgreen}3}$ are
repeated in the tree (even in the small bit of the tree depicted
here), at the same level or at multiple levels; their descendant
subtrees can be identified to form another directed graph
representation of this game of \nim.  Every leaf of the tree
corresponds to the position
${\color{blue}0}{\color{darkpurple}0}{\color{darkgreen}0}$.
\end{example}

\begin{remark}\label{r:impartial}
It is natural to wonder why the games in Definition~\ref{d:gameGraph}
are called ``impartial''.  The term is meant to indicate that the
players have the same options available from each position, as opposed
to \emph{partizan} games, where the players can have distinct sets of
options.  ``But only one player is allowed to play from each
position,'' you may argue, ``so how can you tell the difference
between impartial and partizan games?''  The answer is to put two
games side by side; this is called the \emph{disjunctive sum} of the
two games: the player whose turn it is chooses one game and plays any
of their legal moves in that game.  The result is that a player can
end up making consecutive moves in a single game, if the intervening
move took place in the other game.  In a partizan game, such as
ordinary \chess, white's options from any given position are different
from black's options.  When white makes two consecutive moves on the
same board, they are two consecutive moves \emph{of white pieces
only}.  In contrast, in an impartial game such as \nim, a move by
either player could have been made by the other player, if the other
player had the chance.
\end{remark}

Beyond the ability to distinguish between impartial and partizan
games, what are disjunctive sums good for?  The answer is that these
sums arise naturally when positions decompose, in the course of play,
into smaller independent subgames.

\begin{example}\label{e:dawson1+3}
The final position in Example~\ref{e:dawson} can be represented as a
disjoint union of two \dawson\ boards, one with one file and one
with~three,
$$%
\begin{array}{rcl}
\\[-2ex]
\begin{array}{@{}c@{}}\includegraphics[height=17mm]{dawson6.eps}\end{array}
&\quad=\quad&
\begin{array}{@{}c@{}}\includegraphics[height=17mm]{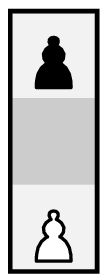}\end{array}
\quad\; + \;\quad
\begin{array}{@{}c@{}}\includegraphics[height=17mm]{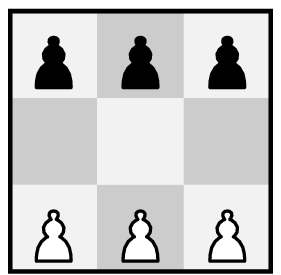}\end{array}
\end{array}
$$
except that now it is black's turn to move.  Similarly, an initial
move at either end of the board obliterates the two columns at that
end, leaving the other player to move.  In addition, any move on a $3
\times 1$ or $3 \times 2$ board obliterates the entire board, as does
a move on the middle file of a $3 \times 3$ board.

This description implies that \dawson\ is a heap game, by restricting
to the ordinary (non-capturing) pawn moves.  Indeed, any connected
board in its initial position is a heap whose size is its number of
files (that is, its width), and any position between bloodbaths is a
disjunctive sum of such boards.  The rules allow any player~to
\begin{itemize}
\item%
eliminate any heap of size $1$, $2$, or $3$;
\item%
take two or three from any heap of size at least~$3$; or
\item%
split any heap of size $d \geq 3$ into two heaps of sizes $k$ and
$d-3-k$.
\end{itemize}
\end{example}

Dawson's choice for the ending of his fairy chess game was both
unfortunate and fortuitous.  It was unfortunate because it made the
game hard: over three quarters of a century after Dawson published his
little game, its solution remains elusive, both computationally and in
a closed form akin to Bouton's \nim\ solution.

\begin{prob}\label{o:dawson}
Determine a winning strategy for \dawson\ and find a polynomial-time
algorithm to calculate it.
\end{prob}

\noindent
Dawson's choice was fortuitous because the simple change of ending
uncovered a remarkable phenomenon: the vast difference between trying
to lose and trying~to~win.

\begin{defn}
Given a finite combinatorial game, the \emph{mis\`ere play} version
declares the winner to be the player who does not move last.
\end{defn}

\noindent
Thus mis\`ere play is what happens when both players try to lose under
the \emph{normal play} rules.  It fosters amusing titles such as
``Advances in losing'' \cite{plambeck09}.  \dawson\ motivated
substantial portions of the development of CGT over the past
few~decades.

Mis\`ere games are generally much more complex than their normal-play
counterparts.  Heuristically, the reason is that, in contrast to the
unique ``zero position'' in normal play, the multiple ``penultimate
positions'' that become winning positions in mis\`ere play cause
ramifications in positions expanding farther from the zero position,
and these ramifications interfere with one another in relatively
unpredictable ways.

Regardless of the reason, aspects of the fact of mis\`ere difficulty
were formalized by Conway in the 1970s (see \cite{ONAG}).  There are,
for example, many more non-isomorphic impartial mis\`ere games than
impartial normal play games of any given \emph{birthday} (the height
of the game tree).  For comparison, note that a complete structure
theory for normal play games was formulated in the late 1930s
\cite{sprague36,grundy39}.  It is based on the \emph{Sprague--Grundy
theorem}, building on Bouton's solution of \nim\ by reducing all
finite impartial games to~it: every impartial game under normal play
is, in a precise sense, equivalent to a single \nim\ heap of some
size.  (The details of this theory would be more appropriate for a
focused exposition on the foundations of CGT, such as Siegel's highly
recommended lecture notes \cite{siegel06}, which proceed quickly to
the substantive aspects from an algebraic perspective.  Additional
background and details can be found in \cite{lessons,winningWays1}.)
Because the ``zero position'' is declared off-limits in mis\`ere
structure theory, the elegant additivity of normal play under
disjunctive sum fails for mis\`ere play, and what results is
algebraically complicated in that case; see Section~\ref{s:misere}.

\subsection{Lattice games}\label{s:lattice}

Impartial combinatorial games admit a reformulation in terms of
lattice points in polyhedra.  For the purpose of Open
Problem~\ref{o:dawson}, the idea is to bring to bear the substantial
algorithmic theory of rational polyhedra \cite{bw03}.  The
transformation begins by a simple change of perspective on \nim,
\dawson, and other heap games.

\begin{example}\label{e:nimLattice}
For certain families of games, the game tree is an inefficient
encoding.  For heap games, it is better to arrange the numbers
of heaps of each size into a nonnegative integer vector whose
$i^\th$ entry is the number of heaps of size~$i$.  Thus the
${\color{blue}3}{\color{darkpurple}7}{\color{darkgreen}4}$ and
${\color{blue}3}{\color{darkpurple}7}{\color{darkgreen}3}$
positions from Example~\ref{e:nimTree} become
$$%
\begin{array}{l}
  {\color{blue}3}\ {\color{darkpurple}7}\ {\color{darkgreen}4}
  \leftrightarrow
  (0,0,{\color{blue}1},{\color{darkgreen}1},0,0,{\color{darkpurple}1}) \in \NN^7
\\
  {\color{blue}3}\ {\color{darkpurple}7}\ {\color{darkgreen}3}
  \leftrightarrow
  (0,0,{\color{turquoise}2},0,0,0,{\color{darkpurple}1}) \in \NN^7.
\end{array}
$$
The moves ``make a heap of size~$j$ into a heap of size $i < j$'' and
``remove a heap of size~$j$'' correspond to other (not necessarily
positive) integer vectors, namely
$$%
\begin{array}{r@{\ }c@{\ }l}
  (\ldots,1,\ldots,-1,\ldots) &=& e_i - e_j \text{ for } i < j, \text{ and}
\\         (\ldots,-1,\ldots) &=& -e_j \text{ for all } j \geq 1,
\end{array}
$$
where $e_1,\ldots,e_d$ is the standard basis of~$\ZZ^d$.  All entries
in the moves are zero except for the $1$ and~$-1$ entries indicated.
\nim\ looks curiously like it could be connected to root systems of
type~$A$, but nothing has been made of this connection.
\end{example}

Example~\ref{e:nimLattice} says that \nim\ positions with heaps of
size at most~$d$ are points in~$\NN^d$, and moves between them are
vectors in~$\ZZ^d$.  The idea behind lattice games is to polyhedrally
formalize the relationship between the positions and moves.  To that
end, for the rest of this section, fix a pointed rational cone $C
\subseteq \ZZ^d$ of dimension~$d$, and write $Q = C \cap \ZZ^d$ for
the normal affine semigroup of integer points in~$C$.  As in earlier
sections, basic knowledge of polyhedra is assumed; see \cite{ziegler}
for additional background.  For simplicity, this survey restricts
attention to lattice games that are played on (the lattice points in)
cones instead of arbitrary polyhedra, and to special rule sets for
which every position has a path to the origin (cf.\
\cite[Lemma~3.5]{latticeGames}); see \cite[\S2]{latticeGames} for full
generality.

Briefly, a lattice game is played by moving a token on a game board
comprising all but finitely many of the lattice points in a polyhedral
cone.  The allowed moves come from a finite rule set consisting of
vectors that generate a pointed cone containing the game board cone.
The game ends when no legal moves are available; the winner is the
last player to move.  The mis\`ere condition is encoded by the
finitely many disallowed lattice point positions.  To define these
properly, it is necessary to define rule sets first.

\begin{defn}\label{d:ruleset}
A \emph{rule set} is a finite subset $\Gamma \subset \ZZ^d \minus
\{0\}$ such that
\begin{enumerate}
\item%
the affine semigroup $\NN\Gamma$ is \emph{pointed}, meaning that its
unit group is trivial, and
\item%
every lattice point $p \in Q$ has a \emph{$\Gamma$-path to~$0$}
in~$Q$, meaning a sequence
$$%
  0 = p_0, \ldots, p_r = p \text{ in } Q, \text{ with } p_{i+1} - p_i
  \in \Gamma,
$$
as illustrated in the following figure.

$$%
\psfrag{0}{\small$0$}
\psfrag{p}{\small$p$}
\psfrag{C}{\small$\!Q$}
\psfrag{NG}{\small{\color{darkgreen}$\NN\Gamma$}}
\begin{array}{@{}c@{}}\\[-2.3ex]\includegraphics{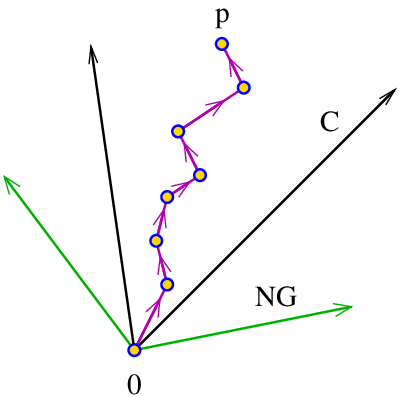}\end{array}
$$
\end{enumerate}
\end{defn}

With these conventions, moves correspond to elements of $-\Gamma$
rather than of~$\Gamma$ itself.  The sign is a choice that must be
made, and neither option is fully convenient.  The choice in
Definition~\ref{d:ruleset} prevents unpleasant signs in the next
lemma.

\begin{lemma}\label{l:NG}
$\NN\Gamma$ contains $Q$ and induces a partial order on~$\ZZ^d$ in
which $p \preceq q$ whenever $q - p \in \NN\Gamma$.
\end{lemma}
\begin{proof}
The containment is immediate from Definition~\ref{d:ruleset}.  The
partial order occurs because $\NN\Gamma$ is a pointed affine
semigroup.
\end{proof}

\begin{defn}\label{d:latticegame}
A \emph{lattice game} played on a normal affine semigroup $Q = C \cap
\ZZ^d$~has
\begin{itemize}
\item%
a rule set $\Gamma$,

\item%
\emph{defeated positions} $D \subseteq Q$ that constitute a finite
$\Gamma$-order ideal, and

\item%
\emph{game board} $B = Q \minus D$.
\end{itemize}
\end{defn}

Lattice points in~$Q$ are referred to as \emph{positions}; note that
these might lie off the game board.  A position $p \in Q$ has a
\emph{move to~$q$} if $p - q \in \Gamma$; the move is \emph{legal} if
$q \in B$.  The order ideal condition, which means by definition that
$q \in D \implies p \in D$ if $p \preceq q$, guarantees that a legal
move must originate from a position on the game board~$B$.

\begin{example}
A heap game in which the heaps have size at most~$d$ is played on $Q =
\NN^d$.  Under normal play, the game board is all of~$\NN^d$, so $D =
\nothing$.  To get mis\`ere play, let the set of defeated positions be
$D = \{0\}$, so $B = \NN^d \minus \{0\}$.  Larger sets of defeated
positions allow generalizations of mis\`ere play not previously
considered.
\end{example}

\begin{example}\label{e:dawsonLattice}
\dawson\ is a lattice game on~$\NN^d$ when the heap sizes are bounded
by~$d$.  The game board in this case is $\NN^d \minus \{0\}$,
corresponding to mis\`ere play.  The rule set is composed of the
following vectors, by Example~\ref{e:dawson1+3}:
\begin{itemize}
\item%
$e_1$;
\item%
$e_2$ and $e_j - e_{j-2}$ for $j \geq 3$;
\item%
$e_3$ and $e_j - e_{j-3}$ for $j \geq 4$ and $e_j - e_i - e_{j-3-i}$
for $j \geq 5$.
\end{itemize}
\end{example}

Historically, the abstract theory of combinatorial games was developed
more using set theory than combinatorics.  Formally, a finite
impartial combinatorial game is often defined as a set consisting of
its options, each being, recursively, a finite impartial combinatorial
game.  Using this language, the disjunctive sum of two games $G$
and~$H$ is the game $G + H$ whose options comprise the union of $\{G'
+ H \mid G'$ is an option~of~$G\}$ and $\{G + H' \mid H'$ is an option
of~$H\}$.  A set of games is \emph{closed} if it is closed under
taking options and under disjunctive sum.  In particular, the
\emph{closure} of a single game~$G$ is the free commutative monoid
on~$G$ and its \emph{followers}, meaning the games obtained
recursively as an option, or an option of an option, etc.  See
\cite{misereQuots} and its references for more details on closure and
on the historical development of CGT.

\begin{thm}\label{t:arbitrary}
Any position in a lattice game determines a finite impartial
combinatorial game.  Conversely, the closure of an arbitrary finite
impartial combinatorial game, in normal or mis\`ere play, can be
encoded as a lattice game played on~$\NN^d$.
\end{thm}
\begin{proof}
The first sentence is a consequence of Lemma~\ref{l:NG}.  The second
is \cite[Theorem~5.1]{latticeGames}: if the game graph has $d$ nodes,
then lattice points in $Q = \NN^d$ correspond to disjunctive sums of
node-positions.
\end{proof}

The proof of Theorem~\ref{t:arbitrary} clarifies an important point
about the connection between lattice games and games given by graphs:
lattice game encodings are efficient only when the nodes of the graph
represent ``truly different'' positions.

\begin{example}
The encoding of the
${\color{blue}3}{\color{darkpurple}7}{\color{darkgreen}4}$
\nim\ game in Example~\ref{e:nimTree} by using all~$d$ of the
followers of
${\color{blue}3}{\color{darkpurple}7}{\color{darkgreen}4}$ as
coordinate directions in~$\NN^d$ is woefully inefficient.  On
the other hand, the
${\color{blue}3}{\color{darkpurple}7}{\color{darkgreen}4}$
position is encoded efficiently in $\NN^7$ because it lies in
the closure of the single \nim\ heap of size~$7$, whose
followers are ``truly different'' from one another.
\end{example}

This explains part of the reason for allowing arbitrary normal affine
semigroups as game boards: more classes of combinatorial games beyond
heap games can be encoded efficiently.  That said, heap games are
now---and have been for decades---key sources of motivation and
examples.  As such, the encoding of heap games is particularly
efficient for the following class of games
\cite[\S6--\S7]{latticeGames}.

\begin{defn}
A lattice game is \emph{squarefree} if it is played on $\NN^d$ and the
maximum entry of any vector in the rule set is~$1$.  Equivalently, a
squarefree game represents a heap game in which each move destroys at
most one heap of each size.
\end{defn}

Multiple heaps of different sizes can be destroyed, and a destroyed
heap can be replaced with multiple heaps of other sizes, as long as
the moves still form a rule~set.

\begin{example}\label{e:octal}
Squarefree games are the natural limiting generalizations of
\emph{octal games}, invented by Guy and Smith \cite{gvalues} with
\dawson\ as a motivating example.  For each $k \in \NN$, an octal game
specifies whether or not any heap of size
\begin{itemize}
\item%
$k$ may be destroyed;
\item%
$j \geq k + 1$ may be turned into a heap of $j - k$; and
\item%
$j \geq k + 2$ may be turned into two heaps of sizes summing to $j-k$.
\end{itemize}
These constitute three binary choices, and hence are conveniently
represented by an octal digit $0 \leq d_k \leq 7$.  \dawson\ is
``$.137$'' as an octal game, where the digits correspond, in order, to
the types of moves in Example~\ref{e:dawsonLattice}.  For example, $7
= 111$ in binary indicates that all three options are allowed for $k =
3$, while $3 = 011$ indicates that only the top two are available for
$k = 2$.  The dot in ``$.137$'' is just a place-holder.
\end{example}

\subsection{Rational strategies}\label{s:ratstrat}

What does a strategy for a combinatorial game look like?  Abstractly,
the finiteness condition ensures that one of the players can force a
win.  The argument explaining why is recursive and elementary.  But
how does one describe such a strategy?  Lattice games provide
malleable data structures for this purpose.

\begin{defn}
Two subsets $W,L \subseteq B$ are \emph{winning} and \emph{losing}
positions for a lattice game with game board~$B$ if
\begin{itemize}
\item%
$B = W \cupdot L$ is the disjoint union of~$W$ and~$L$; and
\item%
$(W + \Gamma) \cap B = L$.
\end{itemize}
\end{defn}

Winning positions are the desired spots to move to.  The first
condition says that every position in~$B$ is either a winning position
(the player who moved to that spot can force a win) or a losing
position (the player who moves from that spot can force a win by
moving to a winning position).  The second condition says that losing
positions are precisely those with (legal) moves to winning positions.

\begin{example}\label{e:nim2}
Consider the game \nim$_2$ of \nim\ with heaps of size at most~$2$.
The rule set in this case is $\bigl\{[\twoline 10], [\twoline 01],
[\twoline {-1}{\phantom-1}]\bigr\}$.  The negatives of these
vectors---representing the legal moves from a generic position---are
depicted in the following figure, along with the winning positions in
\nim$_2$ for both for normal and mis\`ere play.  An easy way to verify
that the depicted sets~$W$ are forced is to figure out what happens on
the bottom row first, from left to right, and then proceed upward, row
by row.
$$%
\psfrag{misere play}{\small \hspace{-1ex}mis\`ere play}
\psfrag{normal play}{\small \hspace{-2ex}normal play}
\psfrag{D}{\small$D$}
\psfrag{NG}{\small\hspace{-.75ex}$-\Gamma$}
\begin{array}{@{}c@{}}\\[-2.3ex]\includegraphics{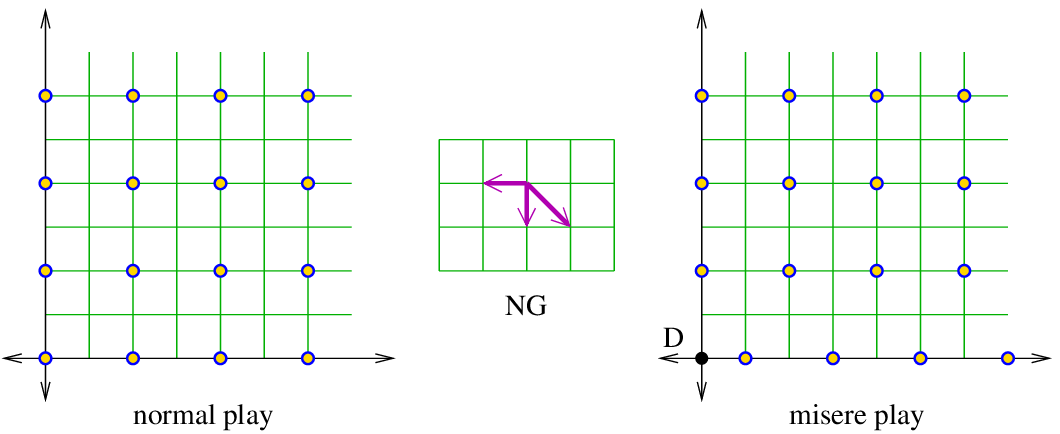}\end{array}
$$
The defeated position, labeled by~$D$ in the mis\`ere play diagram,
causes the bottom row of winning positions to be shifted over one unit
to the right.
\end{example}

\begin{remark}
The disarray caused by the defeated position in Example~\ref{e:nim2}
becomes substantially worse with more complicated rule sets in higher
dimensions.  Much of the study of mis\`ere combinatorial games amounts
to analyzing, quantifying, computing, and controlling the disarray.
\end{remark}

\begin{thm}\label{t:uniqueness}
Given a lattice game with rule set $\Gamma \subset \ZZ^d$ and game
board~$B$, there exist unique sets $W$ and~$L$ of winning and losing
positions for~$B$.
\end{thm}
\begin{proof}
This is \cite[Theorem~4.6]{latticeGames}.  The main point is that the
cones generated by $Q$ and~$\NN\Gamma$ point in the same direction, so
recursion is possible after declaring the $\Gamma$-minimal positions
in~$B$ to be winning.
\end{proof}

Thus everything there is to know about a lattice game is encoded by
its set of winning positions: given a pointed normal affine semigroup,
Theorem~\ref{t:uniqueness} implies that specifying a rule set and
defeated positions is the same as specifying a valid set of winning
positions, at least abstractly.  But the rule set encodes winning
strategies only implicitly, while the set of winning positions---or
better, a generating function $f_W(\ttt) = \sum_{w \in W} \ttt^w$ for
the winning positions---encodes the strategy explicitly.  The
following is \cite[Conjecture~8.5]{latticeGames}.

\begin{conj}\label{c:latticeGames}
Every lattice game has a \emph{rational strategy}: a generating
function for its winning positions expressed as a ratio of polynomials
with integer~coefficients.
\end{conj}

\begin{example}\label{e:nim2'}
Resume Example~\ref{e:nim2}.  In normal play \nim$_2$, a rational
strategy is
$$%
  f_W(a,b) = \frac{1}{(1-a^2)(1-b^2)},
$$
the rational generating function for the affine semigroup~$2\NN^2$.
In mis\`ere play, a rational strategy is
$$%
  f_W(a,b) = \frac{a}{1-a^2} + \frac{b^2}{(1-a^2)(1-b^2)},
$$
where the first term enumerates the odd lattice points on the
horizontal axis, and the second enumerates the normal play winning
positions that lie off the horizontal axis.
\end{example}

\begin{example}\label{e:squarefree}
For a squarefree game, if\/ $W_0 = W \cap \{0,1\}^d \subseteq \NN^d$
then
$$%
  W = W_0 + 2\NN^d.
$$
This is \cite[Theorem~6.11]{latticeGames}.  It implies that the game
has a rational strategy
$$%
  f_W(\ttt) = \sum_{w\in W_0} \frac{\ttt^w}{(1-t_1^2)\cdots(1-t_d^2)}.
$$
The reader is encouraged to reconcile this statement with
Example~\ref{e:nim}.
\end{example}

A rational strategy has a reasonable claim to the title of ``solution
to a lattice game'' because it has the potential to be compact, and it
can be manipulated algorithmically.

\begin{thm}\label{t:rational}
A rational strategy for a lattice game produces algorithms~to
\begin{itemize}
\item%
determine whether a position is winning or losing, and
\item%
compute a legal move to a winning position, given any losing
position.
\end{itemize}
These algorithms are efficient when the rational strategy is a
\emph{short} rational function, in the sense of Barvinok and Woods
\cite{bw03}.
\end{thm}
\begin{proof}
This is a straightforward application of the theory developed by
Barvinok and Woods; see \cite{algsCGT} for details.
\end{proof}

The efficiency in the theorem is in the sense of complexity theory.
Short rational generating functions have not too many terms in their
numerator and denominator polynomials.  They are algorithmically
efficient to manipulate and---when they enumerate lattice points in
polyhedra---to compute.  Since computations of lattice points in
rational polyhedra are efficient, it would be better to get a
polyhedral decomposition of the set~$W$ of winning positions.  In
fact, examples of lattice games exhibit a finer structural phenomenon
than is indicated \cite[Conjecture~8.9]{latticeGames} \&
\cite{commAlgCGT}.

\begin{conj}\label{c:affineStrat}
Every lattice game has an \emph{affine stratification}: an expression
of its winning positions as a finite union of translates of affine
semigroups.
\end{conj}

Roughly speaking, winning positions should be finite unions of sets of
the form (lattice $\cap$ cone).  This definition of affine
stratification differs from \cite[Definition~8.6]{latticeGames} but is
equivalent \cite[Theorem~2.6]{affineStrat}; it would also be
equivalent to require the union to be disjoint, or (independently of
disjointness) the affine semigroups to~be~normal.

\begin{example}\label{e:nim2''}
Consider again the situation from Examples~\ref{e:nim2}
and~\ref{e:nim2'}.  An affine stratification for this game is $W =
2\NN^2$; that is, the entire set of winning positions forms an affine
semigroup.  In mis\`ere play, $W = \big((1,0) + \NN(2,0)\big) \cupdot
\big((0,2) + 2\NN^2\big)$ is the disjoint union of $W_1 = 1 + 2\NN$
(along the first axis) and $W_2$, which equals the translate by twice
the second basis vector of the affine semigroup~$2\NN^2$.
\end{example}

\begin{remark}
Conjecture~\ref{c:affineStrat} bears a resemblance to statements about
local cohomology of finitely generated $\ZZ Q$-graded modules~$M$ over
an affine semigroup ring~$\kk[Q]$ with support in a monomial ideal:
the local cohomology $H^i_I(M)$ is supported on a finite union of
translates of affine semigroups \cite{injAlg}.  If
Conjecture~\ref{c:affineStrat} is true, then perhaps it would be
possible to develop a homological theory for winning positions in
combinatorial games that explains why.
\end{remark}

\begin{thm}\label{t:strat}
A rational strategy can be efficiently computed from any given affine
stratification.
\end{thm}
\begin{proof}
See \cite[\S5]{algsCGT}.  As with Theorem~\ref{t:rational}, this is
reasonably straightforward, applying the methods from \cite{bw03} with
care.
\end{proof}

Algorithms for dealing with affine stratifications and rational
strategies are stepping stones toward a higher aim, which would be to
prove the existence of affine stratifications
(Conjecture~\ref{c:affineStrat}), and hence rational strategies
(Conjecture~\ref{c:latticeGames}), in an efficient algorithmic manner
and in enough generality for \dawson\ (Open Problem~\ref{o:dawson}).
Part of the problem to overcome for \dawson\ is the need to deal with
increasing heap size~$d$.  That problem is key, since the algorithm
for \dawson\ is supposed to be polynomial in~$d$ as $d \to \infty$,
but affine stratifications and rational strategies at present are
designed for fixed~$d$.

\subsection{Mis\`ere quotients}\label{s:misere}

The development of lattice games and rational strategies outlined in
the previous subsections were motivated by---and continue to take cues
from---exciting recent advances in mis\`ere theory by Plambeck and
Siegel \cite{plambeck05, misereQuots} pertaining to mis\`ere quotients
(Definition~\ref{d:misere}).  The lattice point methods are also
beginning to return the favor, spawning new effective methods for
mis\`ere quotients.  This final subsection on combinatorial games ties
together the lattice point perspectives on games and binomial ideals
with mis\`ere quotients, particularly in Theorem~\ref{t:misere}.

\begin{defn}\label{d:misere}
Fix a lattice game with winning positions $W \subseteq Q$ in a pointed
normal affine semigroup.  Two positions $p,q \in Q$ are
\emph{indistinguishable}, written $p \sim q$,~if
$$%
  (p + Q) \cap W - p = (q + Q) \cap W - q.
$$
In other words, $p + r \in W \iff q + r \in W$ for all $r \in Q$.  The
\emph{mis\`ere quotient} of~$G$ is the quotient $\oQ = Q/\til$ of the
affine semigroup~$Q$ modulo indistinguishability.
\end{defn}

Geometrically, indistinguishability means that the winning positions
in the cone above~$p$ are the same as those above~$q$, up to
translation by $p-q$.
$$%
\psfrag{Q}{\small$Q$}
\psfrag{0}{\small$0$}
\psfrag{p}{\small{\color{darkgreen}$p$}}
\psfrag{q}{\small{\color{darkpurple}$q$}}
\psfrag{pvsq}{\small{\color{darkgreen}$p + r$} vs.\ {\color{darkpurple}$q + r$}}
\begin{array}{@{}c@{}}\\[-2.3ex]\includegraphics{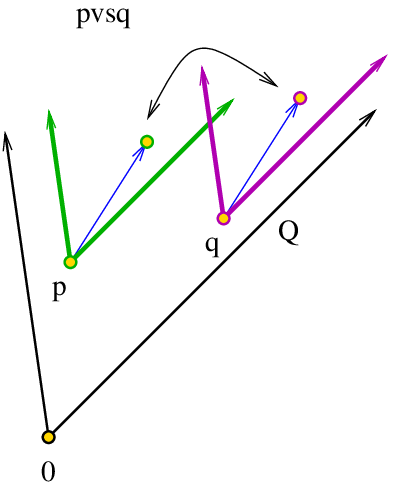}\end{array}
$$

\begin{lemma}
Indistinguishability is a congruence in the sense of
Definition~\ref{d:congruence}, so the mis\`ere quotient~$\oQ$ is a
monoid.\qed
\end{lemma}

\begin{example}\label{e:squarefree'}
In the situation of Example~\ref{e:squarefree}, the mis\`ere quotient
is a quotient of $(\ZZ/2\ZZ)^d = \NN^d/2\NN^d$
\cite[Proposition~6.8]{latticeGames}.  As with
Example~\ref{e:squarefree}, the reader is encouraged to reconcile this
statement with Example~\ref{e:nim}.
\end{example}

\begin{example}\label{e:nim2quotient}
For \nim$_2$ (Examples~\ref{e:nim2}, \ref{e:nim2'},
and~\ref{e:nim2''}), the mis\`ere quotient is the commutative monoid
with presentation $\oQ = \<a,b \mid a^2 =\nolinebreak 1, b^3
=\nolinebreak b\>$, in multiplicative notation.  This monoid has six
elements because it is $\<a \mid a^2 = 1\> \times \<b \mid b^3 = b\>$,
and the second factor has order~$3$.  The presentation of~$\oQ$ can be
seen geometrically in the right-hand figure from Example~\ref{e:nim2}:
translating the grid two units to the right moves $W$ bijectively to
the part of~$W$ outside of the leftmost two columns (this is $a^2 =
1$), and translating the grid up by two units takes the part of~$W$
above the first row bijectively to the part of~$W$ above the third row
(this is $b^3 = b$).
\end{example}

Mis\`ere quotients were introduced by Plambeck \cite{plambeck05} as a
less stringent way to collapse the set of games than had been proposed
earlier by Grundy and Smith \cite{grundy-smith}, in view of Conway's
proof that very little simplification results when the collapsing is
attempted in too large a universe of games \cite[Theorem~77]{ONAG}.
Mis\`ere quotients have subsequently been studied and applied to
computations by Plambeck and Siegel \cite{misereQuots, siegel07}, the
point being that taking quotients often leaves a much smaller---and
sometimes finite---set of positions to consider, when it comes to
strategies.  The Introduction of \cite{misereQuots} contains an
excellent account of the history, including personal accounts from
some of the main~players.

The first contribution of lattice games to mis\`ere theory is the
following.

\begin{prop}
A short rational strategy for a game played on~$Q$ results in an
efficient algorithm for determining the indistinguishability of any
pair of positions in~$Q$.  In particular, an affine stratification
results in such an algorithm.
\end{prop}
\begin{proof}
This is the main result in \cite[\S6]{algsCGT}.  If $f = \sum_{q \in
Q} \phi_q\ttt^q$ and $g = \sum_{q \in Q} \psi_q\ttt^q$ are two short
rational functions, then their \emph{Hadamard product} $f \star g =
\sum_{q \in Q} \phi_q\psi_q\ttt^q$ can be efficiently computed as a
short rational function \cite{bw03}.
\end{proof}

The final result in this section combines three of the main themes
thus far in the survey: lattice games, mis\`ere quotients, and---for
the proof---binomial combinatorics.

\begin{thm}\label{t:misere}
Lattice games with finite mis\`ere quotients have affine
stratifications.
\end{thm}
\begin{proof}
This is \cite[Corollary~4.5]{affineStrat}, given that we are working
with games played on normal affine semigroups.  The proof proceeds via
a general result \cite[Theorem~3.1]{affineStrat} of interest here: the
fibers of any projection $Q \to \oQ$ from an affine semigroup~$Q$ to a
monoid~$\oQ$ all possess affine stratifications.  This is proved using
combinatorial mesoprimary decompositions of congruences---the
(simpler) monoid analogues of binomial mesoprimary
decomposition---whose combinatorics, as in
Example~\ref{e:mesoprimary}, gives rise to affine stratifications of
fibers.  When the mis\`ere quotient of a lattice game is finite, the
winning positions automatically comprise a finite union of~fibers.
\end{proof}

\begin{prob}
Find an algorithm to compute the mis\`ere quotient of any lattice game
starting from an affine stratification.
\end{prob}

Algorithms for computing finite mis\`ere quotients are known and
useful \cite[Appendix~A]{misereQuots}.  In addition, Weimerskirch has
algorithmic methods that apply in the presence of certain known
periodicities \cite{weimerskirch}, although for infinite quotients the
methods fail to terminate.

Binomial primary decomposition, or at least the combinatorial aspects
present in mesoprimary decomposition of congruences on monoids, is
likely to play a role in further open questions on mis\`ere quotients,
including when finite quotients occur, and more complex ``algebraic
periodicity'' questions, which have yet to be formulated
precisely~\cite[Appendix~A.5]{misereQuots}.

\section{Mass-action kinetics in chemistry}\label{s:chem}

Toss some chemicals into a vat.  Stir.  What products are produced?
How fast?  If the process is repeated, can the result differ?  These
questions belong to the study of chemical reaction dynamics.  One of
the earliest theories of such dynamics, the law of mass action, was
formulated by Guldberg and Waage in 1864 \cite{GW1864}.  It is widely
observed to hold in real-life chemical systems (as distinguished from,
say, biochemical systems; see Remark~\ref{r:fails}).

Over the years, mass-action kinetics has matured, especially certain
mathematical aspects following seminal work by Horn, Jackson, and
Feinberg \cite{HJ72,feinberg87} from the 1970s and onward.  In the
past decade, the resulting mathematical formalizations have seen
increasing amounts of algebra, particularly of the binomial sort.
This section provides a brief overview of mass-action kinetics
(Section~\ref{s:mass-action}), covering just enough basics to
understand the relevance of binomial algebra.  {}From there, the main
goal is to explain the Global Attractor Conjecture
(Section~\ref{s:gac}), which posits that a system of reversible
chemical reactions always reaches the same steady state if one exists,
with a view to how binomial primary decomposition could be relevant to
its solution.

Length constraints prevent many substantial details, as well as
examples demonstrating key phenomena, from being included.  For an
elementary introduction to chemical reaction network theory, in
mathematical language, the reader is referred to the well-written
notes by Gunawardena \cite{gun03}.  For details on an abstract
formalization of the law of mass-action in terms of binomials, see
\cite{mass-actionReview}.

\subsection{Binomials from chemical reactions}\label{s:mass-action}

Before presenting mass-action kinetics in general, it is worthwhile to
study a small sample reaction.

\begin{example}\label{e:peroxide}
Consider the breakdown of hydrogen peroxide into water and oxygen:
$$%
  2\peroxide
  \reaction\lambda\mu
  2\water + \oxygen
$$
The $\lambda$ and~$\mu$ here are \emph{rate constants}: $\lambda$
indicates that two molecules of peroxide decompose into two molecules
of water and one molecule of oxygen at some rate, and $\mu$ indicates
that the reverse reaction also occurs, though at another (in this
case, slower) rate: two molecules of water and one molecule of oxygen
react to from two molecules of peroxide.  To be precise, let $x =
[\peroxide]$, $y = [\water]$, and $z = [\oxygen]$ be the
concentrations of peroxide, water, and oxygen in some medium.  These
concentrations are viewed as functions of time, and as such, they
satisfy a system of ordinary differential equations:
\begin{align*}
  \dot x &= 2\mu y^2 z - 2\lambda x^2
\\
  \dot y &= 2\lambda x^2 - 2\mu y^2 z
\\
  \dot z &= \lambda x^2 - \mu y^2 z.
\end{align*}
The right-hand side of $\dot x$ says that after an infinitesimal unit
of time,
\begin{itemize}
\item%
for every two molecules of $\peroxide$ that came together (this is the
$x^2$ term), two molecules of $\peroxide$ disappear (this is the $-2$
in the
coefficient of~$x^2$) some fraction of the time (this is the meaning
of~$\lambda$ in the coefficient of~$x^2$); and
\item%
for every two $\water$ molecules and one~$\oxygen$ that came together
(the $y^2z$ term), two molecules of $\peroxide$ are formed (the $2$
on~$y^2z$) some fraction of the time~($\mu$).
\end{itemize}
The fact that $x$ and~$y$ are squared in all of the right-hand sides
is for the same reason that $y^2$ is multiplied by~$z$: the products
represent concentrations of chemical complexes.  Thus $y^2z$ can be
thought of as an ``effective concentration'' of $2\water + \oxygen$.
\end{example}

\begin{example}
For comparison, it is instructive to see what happens when a new
species is introduced to the chemical equation.  Suppose that the
reaction $2\peroxide \rightleftharpoons 2\water + \oxygen$ in
Example~\ref{e:peroxide} had a fictional additional term on the
right-hand side:
$$%
  2\peroxide
  \reaction\lambda\mu
  2\water + \oxygen + {\color{red}3A}
$$
The right-hand sides of the equations governing the evolution of the
species $\peroxide$, $\water$, and $\oxygen$ would remain binomial:
\begin{align*}
  \dot x &= 2\mu y^2 z {\color{red}a^3} - 2\lambda x^2
\\
  \dot y &= 2\lambda x^2 - 2\mu y^2 z {\color{red}a^3}
\\
  \dot z &= \lambda x^2 - \mu y^2 z {\color{red}a^3}
\end{align*}
as would the new equation ${\color{red}\dot a} = 3\lambda x^2 - 3\mu
y^2 z {\color{red}a^3}$ governing the evolution of the species~$A$.
\end{example}

In general, a chemical reaction involves \emph{species}
$s_1,\ldots,s_n$ with corresponding \emph{concentrations} $[s_i] =
x_i$, each viewed as a function $x_i = x_i(t)$ of time.  In
Example~\ref{e:peroxide}, the species are peroxide, water, and oxygen.
A reaction $A \rightleftharpoons B$ occurs between \emph{chemical
complexes} $A = a_1 s_1 + \cdots + a_n s_n$ and $B = b_1 s_1 + \cdots
+ b_n s_n$.  Thus the complex $A$ is composed of $a_i$ molecules
of~$s_i$ for $i = 1,\ldots,n$, and similarly for~$B$.  In
Example~\ref{e:peroxide}, the complexes are $\peroxide$ and $2\water +
\oxygen$.

\begin{defn}\label{d:mass-action}
A reaction $A \rightleftharpoons B$ between complexes $A = a_1 s_1 +
\cdots + a_n s_n$ and $B = b_1 s_1 + \cdots + b_n s_n$ evolves under
\emph{mass action kinetics} \cite{GW1864} if species $s_i$ is lost at
$a_i$ times a rate proportional to the concentration $x_1^{a_1} \cdots
x_n^{a_n}$ of~$A$, and gained at $b_i$ times a rate proportional to
the concentration $x_1^{b_1} \cdots x_n^{b_n}$ of~$B$.  The reaction
is \emph{reversible} if the reaction rates in both directions are
strictly positive.
\end{defn}

As in Example~\ref{e:peroxide}, the concentration of a complex~$A$ is
the product of the species concentrations with exponents corresponding
to the multiplicities of the species in~$A$.

\begin{prop}\label{p:harpoons}
The differential equation governing the evolution of the reversible
reaction $A \rightleftharpoons B$ from Definition~\ref{d:mass-action}
under mass-action kinetics is
$$%
  \dot x_i = (b_i-a_i)(\lambda \xx^\aa - \mu \xx^\bb),
$$
with rate constants $\lambda, \mu > 0$.  In vector form, with $\lambda
= \lambda_{\aa\bb}$ and $\mu = \lambda_{\bb\aa}$, this becomes
$$%
  \dot\xx = (\bb-\aa)(\lambda_{\aa\bb}\xx^\aa - \lambda_{\bb\aa}\xx^\bb)
$$
\end{prop}
\begin{proof}
This is merely a translation of Definition~\ref{d:mass-action} into
symbols.
\end{proof}

The factor of $\lambda \xx^\aa - \mu \xx^\bb$ in
Proposition~\ref{p:harpoons} is a scalar quantity; the only vector
quantity on the right-hand side is $\bb - \aa$.

A single reaction under mass-action kinetics reaches a steady state
when the binomial on the right-hand side of its evolution equation
vanishes.  Thus the set of steady states for a single reaction is the
zero set of a binomial.  General reaction systems involve more than
one reaction at a time: in a given vat of chemicals, simultaneous
transformations take place involving different pairs of chemical
complexes using the given set of species in the vat.

\begin{defn}\label{d:multiple}
For multiple reactions on a set of species, in which each reaction $A
\rightleftharpoons B$ involves complexes with species vectors $\aa =
(a_1,\ldots,a_n)$ and~$\bb = (b_1,\ldots,b_n)$, the \emph{law of
mass-action} is obeyed if the species evolve according to the
binomial~sum:
$$%
  \dot\xx = \sum_{A \rightleftharpoons B}
  (\bb-\aa)(\lambda_{\aa\bb}\xx^\aa - \lambda_{\bb\aa}\xx^\bb).
$$
\end{defn}

Thus, when more than one reaction is involved, the $i^\th$ entry of
the vector field on the right-hand side is not a binomial but a sum of
binomials, one for each reaction in which species~$s_i$ occurs.
Consequently, the set of steady states need not be
binomial~\cite{balancing}.

\begin{defn}\label{d:balanced}
A point $\xi = (\xi_1,\ldots,\xi_n) \in \RR^n$ is a \emph{detailed
balanced equilibrium} for a reversible reaction system if it is
\emph{strictly positive}, meaning that $\xi_i > 0$ for all~$i$, and
every binomial summand on the right-hand side in
Definition~\ref{d:multiple} vanishes at~$\xi$.  A~system is
\emph{detailed balanced} if it is reversible and has a detailed
balanced equilibrium.
\end{defn}

Definition~\ref{d:balanced} creates the bridge from chemistry to
binomial algebra: the chemical interest lies in equilibria, and these
are varieties of binomial ideals.  Detailed balanced equilibria lie
interior to the positive orthant in~$\RR^n$.  For such concentrations
of the species, each reaction $A \rightleftharpoons B$ in the system
rests at equilibrium with both $A$ and~$B$ present at some nonzero
concentration.  In fact, more is true.

\begin{thm}\label{t:lyapunov}
A detailed balanced equilibrium is a locally attracting steady state.
\end{thm}
\begin{proof}
The main point is that a detailed balanced equilibrium possesses an
explicit strict Lyapunov function (given by Helmholtz free energy)
\cite{HJ72,feinberg87}.
\end{proof}

\begin{remark}\label{r:crnt}
Chemical reaction network theory (CRNT) works in more general settings
than systems of reactions each of which is reversible; see
\cite{gun03} for an introduction.  The theory is most successful when
the system of reactions is \emph{weakly reversible}: each individual
reaction $A \to B$ need not be reversible, but it must be possible to
reach the reactant complex~$A$ from the product complex~$B$ through a
sequence of reactions in the system.  See also Remark~\ref{r:crnt'}.
\end{remark}

\begin{remark}\label{r:fails}
Mass-action kinetics fails for more complicated chemical systems, such
as biochemical ones.  Indeed, it must fail, for life abhors chemical
equilibrium: an organism whose chemical reactions are at steady-state
is otherwise known as dead.  The failure of mass-action kinetics in
biochemical systems occurs for a number of reasons.  For one, the
reaction medium is not homogeneous---that is, the reactants are not
well-mixed).  In addition, the molecules are often too big, and the
number of them too small, for the natural discreteness to be smoothed;
see \cite[\S2]{gun03}.
\end{remark}

\subsection{Global attractor conjecture}\label{s:gac}

Polynomial dynamical systems---linear ordinary differential equations
with polynomial right-hand sides---behave quite poorly and
unpredictably, in general.  The famous chaotic \emph{Lorenz
attractor}, for example, is defined by a vector field whose entries
are simple $3$-term cubics.  However, the binomial nature of
mass-action chemistry lends a striking tameness to the dynamics.

The reversibility hypothesis for detailed balanced systems is natural
from the perspective of chemistry: every reaction can, in principle,
be reversed (although the activation energy required might be
prohibitive under standard conditions).  Overwhelming experience says
that typical chemical reactions---well-mixed, at constant temperature,
as in chemical manufacturing---approach balanced steady states, and
the same products really do emerge every time.  But this is
surprisingly unknown theoretically for detailed balanced systems under
mass-action kinetics, even though their equilibria are local
attractors by Theorem~\ref{t:lyapunov}.

\begin{conj}[Global Attractor Conjecture \cite{HJ72,horn74}]\label{c:gac}
If a reversible reaction system as in Definition~\ref{d:multiple} has
a detailed balanced equilibrium, then every trajectory starting from
strictly positive initial concentrations reaches it in the~limit.
\end{conj}

The Global Attractor Conjecture is ``the fundamental open question in
the field'' \cite[\S1]{mass-actionReview}, since it would close the
book on fundamentally justifying mass-action kinetics.  It is known
that the conjecture holds when the binomial ideal is
prime~\cite{gopalkr09}.  It is also known, for detailed balanced
reaction systems with fixed positive initial species concentrations,
that
\begin{enumerate}
\item%
the detailed balanced equilibrium is unique \cite{feinberg87}, and
\item%
each trajectory tends toward some equilibrium
\cite{sontag01,chavez03}.
\end{enumerate}
Thus it suffices to bound every strictly positive trajectory away from
all boundary equilibria, for then the trajectory limits are forced
toward the detailed balanced one.

\begin{remark}\label{r:crnt'}
Detailed balancing is a stronger hypothesis than required for the
known results listed above, including Theorem~\ref{t:lyapunov}.  Weak
reversibility as in Remark~\ref{r:crnt}, or even a less stringent
condition, often suffices.  Detailed balancing is also weaker than the
hypothesis in the strongest (and still widely believed) form of
Conjecture~\ref{c:gac}, which stipulates a condition called
\emph{complex-balancing} that implies weak reversibility; see, for
instance, \cite[\S4.2]{AS09} for a precise statement.
\end{remark}

What do the boundary equilibria look like?  The restriction of a
detailed balanced reaction system to a coordinate subspace of~$\RR^n$
amounts to forcing the concentrations of some reactant species to be
zero.  Such restrictions still constitute detailed balanced reaction
systems, and the binomials whose vanishing describes the equilibria
come from the binomials in the original system.  This discussion can
be rephrased as follows.

\begin{prop}\label{p:gac}
Boundary equilibria of detailed balanced reaction systems are zeros of
associated primes of the ideal generated by the binomials in
Definition~\ref{d:multiple}.  Conjecture~\ref{c:gac} holds if and only
if every trajectory with positive initial concentrations remains
bounded away from the zero set of every associated prime.\qed
\end{prop}

Proposition~\ref{p:gac} brings binomial primary decomposition to bear
on the chemistry of mass-action kinetics.  The details of how this
occurs are illustrated by certain special cases of
Conjecture~\ref{c:gac} whose proofs are known.  The characterizations
of these cases rely on a polyhedral concept hiding in the dynamics.

\begin{defn}\label{d:stoichiometric}
The \emph{stoichiometric compatibility subspace} of the reaction
system in Definition~\ref{d:multiple} is the real span $S$ of the
vectors $\aa - \bb$ over all reactions $A \rightleftharpoons B$ in the
system.  The \emph{stoichiometric compatibility class} (or
\emph{invariant polyhedron}) of a species concentration vector $\xi
\in \RR^n$ is the intersection of the nonnegative orthant with $\xi +
S$.
\end{defn}

\begin{lemma}\label{l:stoichiometric}
Trajectories for any reaction system are constrained to lie in the
invariant polyhedron of the vector of initial species concentrations.
\end{lemma}
\begin{proof}
This is immediate from the equation for $\dot\xx$ in
Definition~\ref{d:multiple}.
\end{proof}

The known cases of Conjecture~\ref{c:gac} include all systems whose
initial species concentration vectors have invariant polyhedra of
dimension~$2$ or less \cite[Corollary~4.7]{AS09}.  In the language of
Proposition~\ref{p:gac}, the method of proof is to bound all
trajectories away from the zero set of every associated prime whose
intersection with the relevant invariant polyhedron is
\begin{itemize}
\item%
a vertex \cite{dfanderson08,toricDynamicalSystems09} or
\item%
interior to a facet \cite{AS09}.
\end{itemize}
For a comprehensive review of known cases of the Global Attractor
Conjecture, see \cite[\S1 and~\S4]{AS09}.


The combinatorics of binomial primary decomposition might contribute
further than merely the statement of Proposition~\ref{p:gac}.  For
example, mesoprimary decomposition (Definition~\ref{d:mesoprimary};
see \cite{mesoprimary}) provides decompositions of binomial ideals
over the rational or real numbers, and therefore takes steps toward
primary decomposition over the reals.  Mesoprimary decomposition also
characterizes the associated lattices combinatorially, without
a~priori knowing the primary decomposition.  Both could be important
for applications to the Global Attractor Conjecture: perhaps
finiteness conditions surrounding associated lattices
(Example~\ref{e:mesoprimary}) indicates how to produce the desired
trajectory bounds, with the reality (i.e., defined over~$\RR$) of the
components forcing progress away from the boundary, as opposed to
(say) periodicity of~some kind.

In an amazing convergence, graphs associated to \emph{event systems}
\cite[Definition~2.9]{mass-actionReview} provide chemical
interpretations of the graphs $G_I$ from Sections~\ref{s:primary}
and~\ref{s:decomp} (particularly Definition~\ref{d:graph}) in this
survey.  Reversibility of the reaction system means that it is correct
for $G_I$ to be undirected.  The characterization of
\emph{naturality}\/ in \cite[Theorem~5.1]{mass-actionReview} is a
condition on the primary decomposition of the \emph{event ideal}.
This convergence is cause for optimism that lattice-point point
combinatorics will be instrumental in proving Conjecture~\ref{c:gac}
via Proposition~\ref{p:gac}.

\raggedbottom

\end{document}